\documentclass[13pt]{amsart}
     \usepackage{amsmath}
     \usepackage{}
     \usepackage{amsfonts}
     \usepackage{mathrsfs}
     \usepackage{amsfonts, latexsym, rawfonts, amsmath, amssymb, amsthm}
     \usepackage[plainpages=false]{hyperref}

     \usepackage{pdfsync}

     \usepackage[a4paper]{geometry}
     \usepackage{esint}
     \usepackage{graphics, graphicx}
     \usepackage{tikz}
     \usepackage{appendix}

     \numberwithin{equation}{section}

     \newcommand{\beq}{\begin{equation}}
     \newcommand{\eeq}{\end{equation}}
     \newcommand{\beqs}{\begin{eqnarray*}}
     \newcommand{\eeqs}{\end{eqnarray*}}
     \newcommand{\beqn}{\begin{eqnarray}}
     \newcommand{\eeqn}{\end{eqnarray}}
     \newcommand{\beqa}{\begin{array}}
     \newcommand{\eeqa}{\end{array}}

     \def\lra{\longrightarrow}

     \def\bc{\begin{center}}
     \def\ec{\end{center}}

     \def\begeq{\begin{equation}}
     \def\endeq{\end{equation}}
     \def\and{\quad{\rm and}\quad}

     \let\lra=\longrightarrow

     \def\mapright\#1{\, \smash{\mathop{\lra}\limits^{\#1}}\, }

     \newtheorem{prop}{Proposition}[section]
     \newtheorem{theo}[prop]{Theorem}
     \newtheorem{lem}[prop]{Lemma}
     \newtheorem{claim}[prop]{Claim}
     \newtheorem{cor}[prop]{Corollary}
     \newtheorem{rem}[prop]{Remark}
     
     \newtheorem{defi}[prop]{Definition}
     \newtheorem{conj}[prop]{Conjecture}

\begin{document}

      \title{ Rigidity of the Bryant Ricci soliton }

    \author{Ziyi  $\text{Zhao}^{\dag}$ and Xiaohua $\text{Zhu}^{\ddag}$}

\address{BICMR and SMS, Peking
University, Beijing 100871, China.}
\email{ 1901110027@pku.edu.cn\\\ xhzhu@math.pku.edu.cn}

\thanks {$\ddag$ partially supported  by National Key R\&D Program of China  2020YFA0712800 and  NSFC 12271009.}
\subjclass[2000]{Primary: 53E20; Secondary: 53C20,  53C25, 58J05}

\keywords{Bryant  Ricci soliton,   Ricci flow,  curvature pinching,  positive isotropic curvature}

     \begin{abstract}We introduce a new curvature-pinching condition, which is weaker than  the positive sectional curvature or  PIC1,  and then we  prove several rigidity results for rotationally symmetric  solutions of steady Ricci solitons,  i.e.,  the Bryant  Ricci solitons.
   \end{abstract}

      \date{}

   \maketitle

   \tableofcontents

   \setcounter{section}{-1}

 \section{Introduction}

In \cite{Bre1}, Brendle solved the following Perelman's conjecture \cite{P}:

 \textit{Any $3$-dimensional (complete)  non-flat $\kappa$-noncollapsed   steady gradient soliton is rotationally.  Namely, it is isometric to the Bryant  Ricci soliton up to scaling. }

 For  general  dimensional   steady gradient solitons,   Brendle also  proved   \cite{Bre2}

 \begin{theo}\label{brendle-theorem}  Any  steady gradient Ricci soliton  $(M^n,g)$   $(n\geq3)$   with positive sectional curvature must be  isometric to the   Bryant  Ricci soliton up to scaling  if it   is asymptotically cylindrical.

\end{theo}

Since any $3$-dimensional  non-flat steady gradient soliton has   always positive  sectional curvature  \cite{Chen} and  it   is also asymptotically cylindrical  by the $\kappa$-noncollapsed property  \cite{P},   Theorem \ref{brendle-theorem}    in particular  also confirms  the Perelman's conjecture.

 The asymptotically cylindrical  behavior   is an asymptotic  property  for a  Riemannian manifold  $(M, g)$ \cite{Bre2}   (also see  Definition \ref{def1.2}).   In this paper, we purpose to weaken the global  condition of  positive  sectional curvature to generalize Theorem \ref{brendle-theorem}.

 Let   ${\rm Rm}_g$ be the  curvature operator  tensor and $R=R(x)$ the scalar  curvature of $(M, g)$.   We introduce   the following curvature-pinching condition,
 \begin{align}\label{pinching-condition}
  2P<R,
  \end{align}
  where the function  $P=P(x)$ is  defined on $M$ by
  \begin{align}\label{P-defi}
      P(\cdot)=\sup_h\frac{{\rm Rm}_g(h,h)}{|h|^2}=\sup_h\frac{{\rm R}_{ijkl}(\cdot)h^{ik}h^{jl}}{h_{ik}h^{ik}}, ~\forall~{\rm symmetric} ~{\rm 2-tensor}~ h\neq 0.
  \end{align}
The operator  ${\rm Rm}_g(h, \cdot)$  on the space of  symmetric $2$-tensors $h$  arises from the  Lichnerowicz Laplacian  \cite{Bes}  (also see (\ref{LL})).  When the operator is restricted to the traceless $h$,  its negative  is  usually  called   the second  kind curvature operator (cf. \cite{Nis, CGT,  Li}).

The following  is our main result in this paper.

  \begin{theo}\label{main-theorem}  Let  $(M^n,g)$   $(n\geq4)$  be a steady gradient Ricci soliton  which satisfies
$(\ref{pinching-condition})$.  Then  $(M,g)$ must be isometric to the   Bryant Ricci soliton   up to scaling if it   is asymptotically cylindrical.

  \end{theo}



We will  discuss  a class of $(M^n, g)$  ($n\ge 4$)  called  with  PIC0, which  satisfy (\ref{pinching-condition})  in Section 1 (cf. Proposition  \ref{P-condition}).   This class consists of
all  $(M^n, g)$ with  positive sectional curvature or PIC1.
 Thus,    Theorem \ref{main-theorem} is  in particular a  generalization of Theorem \ref{brendle-theorem} for  the steady gradient solitons with   positive sectional curvature.

It seems that  PIC0 is very closely related to PIC (cf. (\ref{PIC0})).   The notion of PIC  was first introduced by Micallef and Moore \cite{MM} in their work on the index   of minimal two-spheres. The PIC condition is preserved under the Ricci flow, proved by Hamilton \cite{Ham} for  $4$-dimension,  and Nguyen \cite{Ng} and Brendle-Schoen \cite{Bre3} independently for  higher  dimensions.
There are   important classification results for   differential structures of compact manifolds  with  PIC  via surgery method,  we refer the reader to    \cite{Ham, CZ1, CZ2, Bre4}, etc.

As an application of Theorem \ref{main-theorem},  we   prove  the following rigidity of the  Bryant soliton  with respect to the linear decay curvature.

 \begin{cor}\label{Curv-decay} Let $(M^n,g)$  $(n\ge 4)$ be a  $\kappa$-noncollapsed steady gradient Ricci soliton with  positive Ricci curvature.   Suppose that  there exists $C_0$ and $r_0$ such that for any  $x\in M$ with  $\rho(x)\geq r_0$ it holds
 \begin{align}\label{decay-condition} {\rm  Rm}_g(x)\ge 0~{\rm and}~ R(x)\leq \frac{C_0}{\rho(x)}.
 \end{align}
 Then $(M,g)$ is    isometric to the Bryant  Ricci soliton if it  satisfies $(\ref{pinching-condition})$.  In particular,  any
  $n$-dimensional    $\kappa$-noncollapsed steady gradient Ricci soliton   with PIC0 must be  isometric to the Bryant Ricci  soliton up to scaling  under the curvature  decay  $(\ref{decay-condition})$.

 \end{cor}

The above corollary  generalizes a result of Deng-Zhu \cite{DZ1} instead of the  positive sectional curvature by  the pinching-curvature (\ref{pinching-condition}) or PIC0.
We note that the global condition  of positive Ricci curvature and  the condition (\ref{decay-condition}) near the infinity are both necessary according to   examples of  $\kappa$-noncollapsed steady gradient Ricci soliton  constructed in  \cite{DW, App},  which are not isometric to the Bryant Ricci  soliton   up to scaling.  But we guess that  the positive Ricci curvature condition can be removed under (\ref{pinching-condition}). We also note  that just assuming the PIC0   is not sufficient to prove  the rigidity of  steady gradient Ricci solitons, thanks for the new examples of Lai about $\kappa$-noncollapsed steady gradient Ricci solitons of   dimension $n\geq 4$  with the positive curvature  operator which are not rotationally symmetric \cite{Yi1}.



 \begin{defi}\label{Symptotic-soliton}  $(M^n,g)$ is called $C^{2, \tau}$-asymptotic to the $n$-dimensional  Bryant  Ricci soliton $(\mathbb R^n, g_0)$ if
there are a ball $B_R(0)\subset  \mathbb R^n$ and a diffeomorphism
 $$F:    (\mathbb R^n\setminus B_R, g_0)   \rightarrow
 ( M, g)  $$
    such that $\hat{g}=F^*g$ satisfies
 \begin{align}\label{asymtotic-behavior}
     |\hat{g}-g_0|_{g_0}(x)\leq O(\rho(x)^{-\tau-\frac{1}{2}}), \quad  |\nabla^k \hat{g}|_{g_0}\leq O(\rho(x)^{-\tau-k}) ,
 \end{align}
 where $k=1, 2, \tau\in(0,1)$.
 \end{defi}

By Corollary  \ref{Curv-decay},    we can  get  the  following rigidity theorem for the steady gradient Ricci solitons.

 \begin{cor}\label{asymtotic-behavior-rigidity}  Any   steady gradient Ricci soliton  with  the curvature condition   $(\ref{pinching-condition})$,  which is  $C^{2, \tau}$-asymptotic to the $n$-dimensional Bryant  Ricci soliton  $(\mathbb R^n, g_0)$,   must be  isometric to  $(\mathbb R^n, g_0)$.

 \end{cor}

 In general, there is no rigidity for the steady soliton under the metric  $C^{l, \tau}$-asymptotic behavior  if there is no curvature restriction.  For examples,   Conlon and Deruelle  have  constructed  a complete steady gradient K\"ahler-Ricci soliton in every K\"ahler class of a crepant resolution over  a Calabi-Yau cone,  which unique up to the flow of the soliton vector field converges at a polynomial rate to the  Cao's  steady gradient K\"ahler-Ricci soliton on the cone  \cite{CoD}.

 It is interesting to mention that a similar version of  Corollary \ref{asymtotic-behavior-rigidity}  has been established for the shrinking gradient Ricci soliton \cite{KW} and the expending  gradient Ricci soliton \cite{De1,De2}, respectively.   We hope that the condition  (\ref{pinching-condition}) can be weakened by the positive Ricci curvature.    Actually, we have the following conjecture.

 \begin{conj}  Any   steady gradient Ricci soliton  with  positive Ricci curvature,  which is  $C^{2, \tau}$-asymptotic to the $n$-dimensional  Bryant Ricci  soliton  $(\mathbb R^n, g_0)$,   must be  isometric to  $(\mathbb R^n, g_0)$.

 \end{conj}

 To prove Theorem \ref{main-theorem},  we will follow  an   approach of Brendle to construct  $\frac{n(n-1)}{2}$ Killing vector fields (VFs) on the steady Ricci soliton $(M^n,g)$ as in the proof of  Theorem \ref{brendle-theorem}  \cite{Bre2},
where  the positivity of  sectional curvature plays a crucial role  in his decay estimate for  the solutions   of  Lichnerowicz-type  equation via the maximum  principle  (also see Remark \ref{remark-brendle-2}).
 However,  in our case with the curvature  condition (\ref{pinching-condition}),  we  use the heat flow method to do the decay estimate  by  solving  the parabolic  version of  Lichnerowicz-type  equation (cf. Section 5).  Thus  we  need to show  that the   Lie derivatives  of  metric $g$ associated to the constructed   $\frac{n(n-1)}{2}$ approximate Killing VFs    have   fast  polynomial decay  (cf. Section 4).

 We now outline the main steps in the proof of Theorem \ref{main-theorem}.

 Step 1 (cf. Section 2, 3). As in \cite{Bre1, Bre2},  we  construct a collection of approximate  Killing VFs $U_{a}, ~a \in \{1,...,\frac{n(n-1)}{2}\}$,  near  the infinity of  $(M, g)$. By  several improvements   of asymptotic estimates for   curvatures  (cf.  Subsection 2.2),  we show that those  $U_{a}$ satisfy    (cf. Proposition \ref{App-Killing-VF}),
\begin{align}\label{u-vector} \left|\mathscr{L}_{U_a}(g)\right|\leq O(\rho^{-\frac{1}{2}-\frac{1}{4}\delta}) ~{\rm and }~|\Delta U_{a}+D_XU_{a}|\leq O(\rho^{-1-\frac{1}{4}\delta}).
\end{align}
The above estimate allows us to   get global solutions of  the following  elliptic equation  associated to  $U_a$  on $(M, g)$,
 \begin{align}\label{vector-equ}\Delta W_a+D_X W_a=0.
 \end{align}
 The method is to
solve   the Dirichlet problem  via  perturbation and exhausting as in \cite[Section 3]{Bre2}.

 Step 2 (cf. Section 4).  We shall do the decay estimate  of  $h_a=\mathscr{L}_{W_a}g$ for solutions of
 Lichnerowicz-type equation associated to the steady Ricci soliton $(M, g)$,
 \begin{align}\label{L-equ}\Delta_L h_a+\mathscr{L}_X(h_a)=0.
 \end{align}
 We prove that there exists a constant $\lambda$ such
 that $|\tilde h_a|$,  where $\tilde h_a=h_a-\lambda \operatorname{Ric}(g)$,  decays to zero at  any  polynomial rate (cf. Theorem \ref{VF-decay}).

 Step 3 (cf. Section 5).    We prove  that the uniqueness of solutions of  (\ref{L-equ}) and
 $\tilde h_a$ is in fact zero then we will  finish the proof of Theorem \ref{main-theorem}.    Our method is to  consider the  following parabolic Lichnerowicz equation for the initial value $\tilde h_a$ constructed in Step 2,
  \begin{align}\label{para-L}\frac{\partial}{\partial t}h=\Delta_{L,g(t)}h,
  \end{align}
  where $g(t)$ is the Ricci flow solution induced by $g$ via transformations  group $\Phi^*_t$  generated by  the soliton VF.
 By modifying  the   pinching estimate  of Anderson-Chow  for the solution of (\ref{para-L})   in $3$-dimension \cite{AC} (also see \cite{Yi2}),   we introduce the following heat-type  equation for any dimension $n\ge 3$,
 \begin{align}\label{H}
     \frac{\partial}{\partial t}H-\Delta_{g(t)} H-2P(\cdot, t) H=0,
 \end{align}
where  $P(\cdot, t)=\Phi^*_t P(\cdot)$ and $P(\cdot)$  is the function defined  by  (\ref{P-defi}).
We will show that there exists a bounded positive super-solution  of  (\ref{H}) (cf. Lemma  \ref{supersolution}),  and  then use the heat kernel method  to prove that  the limit solution of
(\ref{para-L}) is just zero (cf. Proposition \ref{h-vanish}).

Theorem \ref{main-theorem},  Corollary  \ref{Curv-decay} and  Corollary \ref{asymtotic-behavior-rigidity} will be finally  proved in Section 6.

\vskip5mm

$\mathbf {Acknowledgment.}$  The authors would like to thank  Xiaodong Cao and  Xiaolong Li for valuable conversations on the operator  ${\rm Rm}_g(h, \cdot)$ and their recent works on  the Nishikawa's conjecture.

 \section{Preliminary}

A complete Riemannian metric $g$ on  $M$ is  called a gradient Ricci soliton if there exists a smooth function $f$ ( which is called a defining function)  on $M$ such that
\begin{equation}\label{Def-soliton}
R_{ij}(g)+\rho g_{ij}=\nabla_{i}\nabla_{j}f,
\end{equation}
where $\rho\in \mathbb{R}$ is a constant. The gradient Ricci soliton is called expanding,   steady and shrinking  according to the sign $\rho >,  = ,  <0$,  respectively.     These three types  of   Ricci solitons  correspond to three different blow-up solutions  of Ricci flow  \cite{H3}.

 In case of  steady Ricci solitons, we can rewrite (\ref{Def-soliton}) as
  \begin{align}\label{soliton-equation} 2\operatorname{Ric}(g)=\mathscr{L}_Xg,
  \end{align}
  where  $\mathscr{L}_X $ is the Lie operator along   the gradient VF $X = \nabla f$  generalized by $f$.
Let $\{\Phi^*_t\}_{t\in(-\infty,\infty)}$ be   a  1-ps  of transformations   generated by $-X$.   Then
$g(t)=\Phi^*_t(g)$ ($t\in(-\infty,\infty))$ is a solution of Ricci flow. Namely,  $g(t)$ satisfies
\begin{align}\label{ricci-equ}\frac{\partial g}{\partial t}=-2{\rm Ric}(g), ~ g(0)=g.
\end{align}
For simplicity,  we call  $g(t)$  the soliton  Ricci flow of  $(M, g)$.

 By (\ref{soliton-equation}), we have
 $$\langle\nabla R,\nabla f\rangle=- 2\operatorname{Ric}(\nabla f,\nabla f),$$
 where $R$ is the scalar curvature of $g$.
 It follows
 $$R+|\nabla f|^2={\rm Const.}$$
  Since $R$ is alway positive (\cite{Zh, Cao2}),  the above equation can be normalized by
 \begin{align}\label{scalar-equ} R+|\nabla f|^2=1.
 \end{align}

In dimension $n=2$, it is classified that the  ciger  is only  solution \cite{Ham,  CLN}.   Brendle prove that   the Bryant soliton is only  $3$-dimensional  non-flat $\kappa$-noncollapsed   steady gradient soliton up to scaling \cite{Bre1}. Here we call a steady gradient soliton as
 the Bryant Ricci soliton if it is rotational and is normalized by (\ref{scalar-equ}). In this paper, we always  assume $n\ge 4$,  if there is no special emphasis.

The following result has been proved in \cite{Cao2, ChD}.

 \begin{lem}\label{linear-growth}  Let $(M,g)$ be a steady gradient soliton. Suppose that the scalar curvature decays uniformly and have positive Ricci curvature outside of a compact set $K\subset M$. Then there exists positive constants $r_0$, $C_1$, $C_2$ such that for all $x\in M$ with  $\rho(x)\geq r_0$ it holds
 \begin{equation}\label{f-linear}
     C_1\rho(x)\leq f(x)\leq C_2\rho(x),
 \end{equation}
where $\rho(x)=d(p, x)$.
 \end{lem}

\subsection{Asymptotically cylindrical steady Ricci solitons}

Recall (cf. \cite{Bre2})
 \begin{defi}\label{def1.2} We say  that a Riemannian manifold $(M,g)$ is asymptotically cylindrical if the following holds:

(i) The scalar curvature satisfies $\frac{C_1}{\rho(x)}\leq R(x)\leq \frac{C_2}{\rho(x)}$ as $\rho(x)>>1$,  where $C_1$, $C_2$ are two positive constants.

 (ii) Let $p_m$ be an arbitrary sequence of marked points going to infinity. Consider the rescaled metrics
 \[\hat{g}^{(m)}(t)=r_m^{-1}\Phi^*_{r_mt}(g),\]
 where $r_mR(p_m)=\frac{n-1}{2}+o(1)$ as $m\rightarrow \infty$, the flow $(M,\hat{g}^{(m)}(t),p_m)$ converges in the Cheeger-Gromov sense to a family of shrinking cylinders $(\mathbb S^{n-1}\times \mathbb{R},~ \bar{g}(t))$, $t\in (0,1)$. The metric  $\bar g(t)$ is given by
 \begin{align}\label{split-flow}\bar g(t)= (n-2)(2-2t)g_{\mathbb S^{n-1}(1)} +dr^2,
 \end{align}
 where  $\mathbb S^{n-1}(1)$ is the unit sphere in the Euclidean  space.

 \end{defi}

Under the  condition (ii) in Definition   \ref{def1.2},   the following curvature behavior estimate  for the  steady gradient solitons  has been proved in \cite[Lemma 6.5, 6.6]{DZ2}.

 \begin{lem}\label{DZ-6.6} Let $(M,g)$ be a noncompact steady Ricci soliton, which satisfies  the  condition (ii) in Definition   \ref{def1.2}.  For $p\in M$,  let $\{e_1,e_2,...,e_n\}$ be an orthonormal basis of $T_pM$ with respect to metric $g$, and $e_n=\frac{\nabla f}{|\nabla f|}$.   Then
\begin{equation*}
    \begin{aligned}
     &\frac{|\operatorname{Hess} R|+|\Delta \mathrm{Ric(g)}|}{R^2} \rightarrow 0, \quad \text { as } \quad \rho(x) \rightarrow \infty, \\
     &\frac{\operatorname{Ric}\left(e_i, e_j\right)}{R} \rightarrow \frac{\delta_{i j}}{(n-1)}, \quad \text { as } \quad \rho(x) \rightarrow \infty, \\
     &\frac{\operatorname{Ric}\left(e_i, e_n\right)}{R} \rightarrow 0 \text {, as } \quad \rho(x) \rightarrow \infty, \\
     &\frac{\operatorname{Rm}\left(e_k, e_j, e_k, e_l\right)}{R} \rightarrow \frac{(1-\delta_{jk})\delta_{jl}}{(n-1)(n-2)}, \quad \text { as } \quad \rho(x) \rightarrow \infty, \\
     &\frac{\operatorname{Rm}\left(e_n, e_i, e_j, e_k\right)}{R} \rightarrow 0, \quad \text { as } \quad \rho(x) \rightarrow \infty.
     \end{aligned}
 \end{equation*}
 In particular,  there exists a compact set $K\subset M$  such that $(M,g)$ has positive sectional curvature outside of $K$.

 \end{lem}

 By the above lemma   (also see  \cite[Section 3]{ChD}),    the first condition in  Definition  \ref{def1.2} can be deduced from  the second one.   Hence,  we may say  that  $(M,g)$ is asymptotically cylindrical  with only assuming the second condition.

Lemma \ref{DZ-6.6} also implies that that $(M, g)$ has positive sectional curvature outside of a compact set of $M$. Then  by  Lemma \ref{linear-growth},  the defining function $f$ has a linear growth. Thus we may assume $f\geq 4$ on $M$.

\subsection{A class of Riemannian manifolds with    $(\ref{pinching-condition})$}

 \begin{defi}\label{PIC-Ricci}   A Riemannian manifold $(M^n, g)$  $(n\geq 4)$  is said to have PIC1,  if for any $p\in M$, and any orthonormal $4$-frame $\{e_i,e_j,e_k,e_l\}$ of $T_pM$,  we have
\begin{align}\label{PIC1}
     R_{ikik}+R_{jkjk}+\lambda^2(R_{ilil}+R_{jljl})-2\lambda R_{ijkl}>0 ~for~all~\lambda\in[0,1].
 \end{align}

 \end{defi}

   $(M^n, g)$   is of PIC1 is equivalent to that    $(\tilde M^{n+1}=M\times \mathbb R,  \tilde g=g+dx\otimes dx )$ is of   PIC (cf. \cite{Bre3}).  Moreover,   $(M^n,g)$ is of  PIC  if and only if   (\ref{PIC1}) holds for  $\lambda=1$.  Thus   PIC1 implies  PIC.

    In case that   (\ref{PIC1}) holds for  $\lambda=0$, namely,
   \begin{align}\label{PIC0}
    R_{\{i,j,k\}}=R_{ikik}+R_{jkjk}>0,
   \end{align}
    we may regard  (\ref{PIC0})  as   (\ref{PIC1})  for the orthonormal $4$-frame $\{e_i,e_j,e_k, e_0\}$ of $T_{\tilde p}\tilde M$, where $e_0$ is the unit vector field tangent to $\mathbb R$.
 For convenience,     we call $(M^n, g)$ with PIC0 if (\ref{PIC0}) holds for any orthonormal $3$-frame $\{e_i,e_j,e_k\}$ of $T_pM$.
  Clearly,  the positivity of sectional curvature of  $(M^n, g)$ implies PIC0,    and   also PIC1 implies  PIC0.   We note that in case of $n=3$   PIC0 is equivalent to the positivity of Ricci curvature of $g$.

 Recall that the function  $P=\sup_h\frac{{\rm Rm}(h,h)}{|h|^2}$ for any symmetric $2$-tensor $h\neq 0$.  Then at each point  $x$ in $M$, we can choose an  orthonormal  frame  such that $h$ is diagonal and
\begin{align}\label{pic1}
    \frac{{\rm Rm}(h,h)}{|h|^2}(x)=  \frac{\sum_{i,j=1}^nR_{ijij}h_{ii}h_{jj}}{\sum_{i=1}^nh_{ii}^2}.
 \end{align}
 Denote $(M_{ij})=(R_{ijij})$ as a matrix.    Then it  is symmetric  and  its  diagonal elements are all $0$.   Thus  $P$ is bounded by the norm of the largest eigenvalue of $(M_{ij})$. For simplicity, we  let $h_i=h_{ii}$.

 \begin{lem}\label{PIC1-condition}
 If $(M^n,g)$, $n\geq 4$ has  PIC0,   then for any orthonormal $4$-frame $\{e_i,e_j,e_k,e_l\}$ of $T_pM$, $1\leq i< j< k< l\leq n$, and any  symmetric 2-tensor $h$,  it holds
\begin{align}\label{first-part}
    \sum_{a,b\in\{i,j,k,l\}}M_{ab}(h_a-h_b)^2&= M_{ij}(h_i-h_j)^2+M_{ik}(h_i-h_k)^2+M_{il}(h_i-h_l)^2 \notag\\
    &+M_{jk}(h_j-h_l)^2+M_{jl}(h_j-h_l)^2+M_{kl}(h_k-h_l)^2\notag\\
    &\ge 0
\end{align}

 \end{lem}

 \begin{proof}
By (\ref{PIC0}), it easy to see that  there are at most two negative elements  $M_{a'b'}$ and  $M_{c'd'}$  in these $M_{ab}$.  Moreover, the indices  $a', b', c'$ and  $d'$ satisfy
$$\{a', b', c', d' \}=\{i,j,k,l\}.$$
W.L.O.G.,  we may assume $M_{ij}<0$, $M_{kl}<0$. Thus
\[\max\{|M_{ij}|,|M_{kl}|\}<\min\{M_{ik},M_{il},M_{jk},M_{jl}\}=M>0.\]
Hence,
\begin{align}\label{positive}
    & M_{ij}(h_i-h_j)^2+M_{ik}(h_i-h_k)^2+M_{il}(h_i-h_l)^2\notag \\
    &+M_{jk}(h_j-h_l)^2+M_{jl}(h_j-h_l)^2+M_{kl}(h_k-h_l)^2\notag \\
    &\geq M((h_i-h_k)^2+(h_i-h_l)^2+(h_j-h_k)^2+(h_j-h_l)^2\notag\\
    &-(h_i-h_j)^2-(h_k-h_l)^2)\notag\\
    &=M(h_i^2+h_j^2+h_k^2+h_l^2+2h_ih_j+2h_kh_l-2h_ih_k-2h_ih_l-2h_jh_k-2h_jh_l)\notag\\
    &=M(h_i+h_j-h_k-h_l)^2\geq 0.
\end{align}

The relation (\ref{positive}) is  also true for  the case that only one of $M_{ij}$ and $M_{kl}$ is negative.  On the other hand, (\ref{first-part}) is obviously true if all $M_{ab}\geq 0$. Hence, the lemma is proved.

 \end{proof}

\begin{prop}\label{P-condition}   PIC0 implies  $(\ref{pinching-condition})$  when $n\ge 4$.

\end{prop}

 \begin{proof}   To prove $R-2P>0$, we  need to check that for any  non-zero  symmetric  2-tensor $h$ it holds
\begin{align}\label{RP2} \sum_{i<j}M_{ij}(\sum_{k=1}^nh_k^2)-2\sum_{i<j}M_{ij}h_ih_j>0.
\end{align}
Note
\begin{align}\label{R-2P}
    (\sum_{i<j}M_{ij})(\sum_{k=1}^nh_k^2)-2\sum_{i<j}M_{ij}h_ih_j=\sum_{i<j}M_{ij}(h_i-h_j)^2+
    \sum_{i<j}M_{ij}(\sum_{k\neq i,j}h_k^2).
\end{align}
 By Lemma \ref{PIC1-condition},   we have
\begin{align}\label{R-2P-1}
    \sum_{i<j}M_{ij}(h_i-h_j)^2=\frac{2}{(n-2)(n-3)}\sum_{1\leq i<j<k<l\leq n}\sum_{a,b\in\{i,j,k,l\}}M_{ab}(h_a-h_b)^2\geq 0.
\end{align}
On the other hand,
any fixed  $i\in \{1, ..., n\}$,  by  (\ref{PIC0}),  it is easy to see
\begin{align}
4\sum_{j<k\neq i}M_{jk}&=2\sum_{j\neq i} \sum_{k\neq i}M_{jk}\notag\\
&= \sum_{j\neq i} ( R_{\{1,2,j\}} +R_{\{2,3,j\}}+...+ R_{\{i-1,i+1,j\}}+...+ R_{\{n,1,j\}} \notag\\
&>0. \notag
\end{align}
Then
\begin{align}\label{R-2P-2}
    \sum_{i<j}M_{ij}(\sum_{k\neq i,j}h_k^2)=\sum_i^nh_i^2(\sum_{j<k\neq i}M_{jk})>0.
\end{align}
Hence,   combining  $(\ref{R-2P-1})$ and $(\ref{R-2P-2})$,  we  get (\ref{RP2})  from (\ref{R-2P}) immediately.

 \end{proof}

 \section{Fine  estimates   of curvature behavior}

 In this section,  we  modify the argument in \cite[Section 2]{Bre2} to improve the estimate for the curvature behavior around
 level sets of  the steady  Ricci soliton with  asymptotically cylindrical property.  The result will be used to construct  the solutions  of elliptic equation (\ref{vector-equ}) for VFs in Section 3.

For a fixed point $p\in M$,  we choose an orthonormal frame $\{e_1,e_2,...,e_n\}$ with $e_n=\frac{\nabla f}{|\nabla f|}$.
By \cite[Proposition 2.1]{Bre2} (also see  Lemma \ref{DZ-6.6}),  we have
 \begin{align}\label{Ric1}
     \operatorname{Ric}\left(e_i, e_j\right)=\frac{1}{n-1} R \delta_{i j}+o\left(r^{-1}\right), ~i,j\in\{1,2,...,n-1\}.
 \end{align}
 and
 \begin{align}\label{Ric2}
     2 \operatorname{Ric}\left(e_i, X\right)=-\left\langle e_i, \nabla R\right\rangle=o\left(r^{-\frac{3}{2}}\right).
 \end{align}
 Moreover, we have
 \begin{align}\label{Brendle-R-Ric}
     |\operatorname{Ric}(g)|^2-\frac{1}{n-1}R^2=O(r^{-\frac{5}{2}})
    \end{align}
      and
 \begin{align}\label{Ric3}
     2 \operatorname{Ric}(X, X)=-\langle X, \nabla R\rangle=\Delta R+2|\operatorname{Ric}(g)|^2=\frac{2}{n-1} R^2+o\left(r^{-2}\right).
 \end{align}

Let
\begin{align}\label{T}
    T=(n-1)\operatorname{Ric}(g)-Rg+Rdf\otimes df
\end{align}
 be defined as in   \cite[Section 2]{Bre2}. The following estimate was prove by Brendle.

 \begin{lem}\label{Brendle-round}
 If  $(M,g)$ is asymptotically cylindrical,  it holds
 \begin{align}\label{Brendle-T}
     |T|\leq O(r^{-\frac{3}{2}}),~ |\nabla T|\leq O(r^{-\frac{3}{2}-16\delta}).
 \end{align}
As a consequence,
 \begin{align}\label{Brendle-R}
     |\nabla R|\leq O(r^{-\frac{3}{2}-16\delta}),~|\Delta R|\leq O(r^{-2-8\delta}),
 \end{align}
and
 \begin{align}\label{Brendle-Ric}
   &  |\nabla \operatorname{Ric(g)}|\leq O(r^{-\frac{3}{2}-16\delta}),~ |D^2\operatorname{Ric(g)}|\leq O(r^{-2-8\delta}),\notag\\
   & |D^3\operatorname{Ric(g)}|\leq O(r^{-\frac{5}{2}-4\delta}).
 \end{align}
 \end{lem}

The following  curvature estimates   was also  proved  by Brendle.

 \begin{lem}\label{Brendle-round-2}
  \begin{align}\label{Brendle-fR}
     fR=\frac{n-1}{2}+O(r^{-8\delta}),
 \end{align}
 \begin{align}\label{Brendle-fRic}
     f\operatorname{Ric(g)}\leq (\frac{1}{2}+o(1))g,~ f^2\operatorname{Ric(g)}\geq cg, ~c>0,
 \end{align}
 and
 \begin{align}\label{Brendle-fRm}
            f R_{i j k l}\notag &=\frac{1}{2(n-2)}\left(g_{i k}-\partial_i f \partial_k f\right)\left(g_{i k}
            -\partial_i f \partial_k f\right) \notag\\
         &-\frac{1}{2(n-2)}\left(g_{i l}-\partial_i f \partial_l f\right)\left(g_{j k}-\partial_j f \partial_k f\right)\notag\\
         &+O(r^{-8\delta}).
      \end{align}

 \end{lem}

 \subsection{$\frac{1}{2}$-order improvement}

 In this subsection, we will improve the estimates  both  in   Lemma \ref{Brendle-round} and Lemma \ref{Brendle-round-2}.

 First we prove

 \begin{lem}\label{DT}
     \begin{align}\label{derivative-X-T} |D_XT|\leq O(r^{-2-8\delta}) ~{\rm and}~|D_{X,X}^2T|\leq O(r^{-\frac{5}{2}-8\delta}).
     \end{align}
 \end{lem}

 \begin{proof} By (\ref{Brendle-Ric}), we have
  \begin{align}\label{derivative-f}R_{i j k l} \partial^l f=D_i \operatorname{Ric}_{j k}-D_j \operatorname{Ric}_{i k}=O\left(r^{-\frac{3}{2}-16\delta}\right).
  \end{align}
    Then by (\ref{Brendle-R}) and  (\ref{Brendle-Ric}), we get
 \begin{align}
         \begin{aligned}
            & -D_XT_{ik}\notag\\
            &=(n-1)\Delta \operatorname{Ric}_{ik}-\Delta R(g_{ik}-\partial_if\partial_kf)+2(n-1)R_{ijkl}\operatorname{Ric}^{jl} \notag\\
             &-2|\operatorname{Ric}(g)|^2(g_{ik}-\partial_if\partial_kf)-R\operatorname{Ric}(X,e_i)\partial_kf-R\operatorname{Ric}(X,e_k)\partial_if \notag\\
             &=2(n-1)R_{ijkl}\operatorname{Ric}^{jl}-2|\operatorname{Ric}(g)|^2(g_{ik}-\partial_if\partial_kf)+O(r^{-2-8\delta}).\notag
              \end{aligned}
     \end{align}
             By (\ref{soliton-equation}), it follows
              \begin{align}
        & -D_XT_{ik}\notag\\
              &=\frac{2}{(n-1)(n-2)}R^2(n-2|\nabla f|^2+|\nabla f|^4)(g_{ik}-\partial_if\partial_kf)\notag\\
             &-\frac{2}{(n-1)(n-2)}R^2(g_{ik}-\partial_if\partial_kf)+\frac{2}{(n-1)(n-2)}R^2(2-3|\nabla f|^2+|\nabla f|^4)\partial_if\partial_kf\notag\\
             &-\frac{2}{(n-1)^2}R^2(n-2|\nabla f|^2+|\nabla f|^4)(g_{ik}-\partial_if\partial_kf)+O(r^{-2-8\delta})\notag\\
             &=\frac{2}{(n-1)(n-2)}R^2((n-2)+O(r^{-1}))(g_{ik}-\partial_if\partial_kf)\notag\\
             &-\frac{2}{(n-1)^2}R^2((n-1)+O(r^{-1}))(g_{ik}-\partial_if\partial_kf)+O(r^{-2-8\delta})\notag\\
             &=O(r^{-2-8\delta}).\notag
          \end{align}
     Thus the first relation  of (\ref{derivative-X-T})  is true.

     For the second  relation,  we use the fact
 \begin{equation*}
         -D_{X,X}^2T_{ik}-\left\langle D_XX,\nabla T_{ik}\right\rangle=\left\langle X,\nabla(-D_XT_{ik})\right\rangle.
     \end{equation*}
   By (\ref{Brendle-T}),  we see
  \begin{align}\label{D2XT1}
         \left\langle D_XX,\nabla T_{ik}\right\rangle=\operatorname{Ric}(X,e_j)\nabla_jT_{ik}=O(r^{-3}).
     \end{align}
   On the other hand,
 \begin{align}\
          \begin{aligned}
           &  \left\langle X,\nabla(-D_XT_{ik})\right\rangle\notag\\
            &= (n-1)D_X\Delta\operatorname{Ric}_{ik}-D_X\Delta R(g_{ik}-\partial_if\partial_kf)\notag\\
             &+(\Delta R+2|\operatorname{Ric}(g)|^2)\operatorname{Ric}(X,e_i)\partial_kf+(\Delta R+2|\operatorname{Ric}(g)|^2)\operatorname{Ric}(X,e_k)\partial_if\notag\\
             &+2(n-1)D_XR_{ijkl}\operatorname{Ric}^{jl}+2(n-1)R_{ijkl}D_X\operatorname{Ric}^{jl}\notag\\
             &-4D_X\operatorname{Ric}_{jl}\operatorname{Ric}^{jl}(g_{ik}-\partial_if\partial_kf)\\
             &-D_XR(\operatorname{Ric}(X,e_i)\partial_kf+\operatorname{Ric}(X,e_k)\partial_if)\notag\\
             &-R(D_X\operatorname{Ric}(X,e_i)\partial_kf+D_X\operatorname{Ric}(X,e_k)\partial_if)\notag\\
             &-2R\operatorname{Ric}(X,e_i)\operatorname{Ric}(X,e_k).
             \end{aligned}
     \end{align}
           Then by   (\ref{Brendle-R}) and  (\ref{Brendle-Ric}), it follows
        \begin{align}
         \begin{aligned}
          &  \left\langle X,\nabla(-D_XT_{ik})\right\rangle\notag\\
             &=(n-1)D_X\Delta\operatorname{Ric}_{ik}-D_X\Delta R(g_{ik}-\partial_if\partial_kf)+O(r^{-\frac{5}{2}-8\delta})\notag\\
             &=(n-1)\Delta D_X\operatorname{Ric}_{ik}-\Delta D_XR(g_{ik}-\partial_if\partial_kf)
             +Rm*(\nabla \operatorname{Ric}(g)+\nabla R)\notag\\
        &+(n-1)\nabla_j(R_{lji}^m\operatorname{Ric}_{mk}\partial^lf+R_{ljk}^m\operatorname{Ric}_{im}\partial^lf)+O(r^{-\frac{5}{2}-8\delta})\notag\\
             &=O(r^{-\frac{5}{2}-8\delta}).\notag
         \end{aligned}
     \end{align}
     Thus combining this with (\ref{D2XT1}), we get the second  relation of (\ref{derivative-X-T}).

 \end{proof}

Next we improve the $C^0$-norm of $T$-tensor by the above lemma.

 \begin{prop}\label{Improve-T}
 \begin{align}\label{T-norm} |T|\leq O(r^{-\frac{3}{2}-4\delta}).
 \end{align}
 \end{prop}

 \begin{proof}
 We  use the  parabolic maximum principle to get the estimate  (\ref{T-norm})  of $T$ as in \cite[Section 2]{Bre2}.
We note that by the Ricci equation (\ref{ricci-equ})
 the Ricci tensor satisfies the equation
 \[\Delta \operatorname{Ric}_{i k}+D_X \operatorname{Ric}_{i k}=-2 \sum_{j, l=1}^n R_{i j k l} \operatorname{Ric}^{j l}\]
 Then by  the identity
 $$\Delta X+D_X X=0, $$
    we obtain
 \begin{align}\label{Ric-DF}
     &\Delta\left(R g_{i k}-R \partial_i f \partial_k f\right)+D_X\left(R g_{i k}-R \partial_i f \partial_k f\right) \notag\\
     &=(\Delta R+\langle X, \nabla R\rangle)\left(g_{i k}-\partial_i f \partial_k f\right) \notag\\
     &-2 \operatorname{Ric}_{ij}\partial_jR\partial_k f-2\operatorname{Ric}_{jk}\partial_jR\partial_if
     -2R\operatorname{Ric}_{ij}\operatorname{Ric}_{jk} \notag\\
     &-R\Delta \partial_if\partial_kf-R\Delta\partial_kf\partial_if
      -RD_X\partial_if\partial_kf-RD_X\partial_kf\partial_if \notag\\
     &=-2|\operatorname{Ric}(g)|^2\left(g_{i k}-\partial_i f \partial_k f\right)+O\left(r^{-\frac{5}{2}-16\delta}\right).
\end{align}
 In the last relation we used (\ref{derivative-f}).

 By (\ref{Ric1}) and  (\ref{Ric2}) together with (\ref{derivative-f}), we have
 \[\Delta\partial_if\partial_kf=\partial_i\Delta f+R_{jijl}\partial^lf=\partial_iR+R_{jijl}\partial^lf\leq O(r^{-\frac{3}{2}-16\delta}).\]
 Then
 \begin{equation*}
     \begin{aligned}
     &\Delta T_{i k}+D_X T_{i k}\\
      &=-2 \sum_{j, l=1}^{n-1} R_{i j k l} T^{j l}-2 \sum_{j=1}^{n}R_{ijkn}\operatorname{Ric}^{jn}-2\sum_{l=1}^{n}R_{inkl}\operatorname{Ric}^{nl}+2R_{inkn}\operatorname{Ric}^{nn}\\
     &-2 R \operatorname{Ric}_{i k}+2RR_{inkn}+2\sum_{j, l=1}^{n}R_{ijkl}R\partial^jf\partial^lf +2|\operatorname{Ric}(g)|^2\left(g_{i k}-\partial_i f \partial_k f\right)\notag\\
     &+O\left(r^{-\frac{5}{2}-16\delta}\right).
 \end{aligned}
 \end{equation*}
 By (\ref{derivative-f}), it follows
 \begin{align}
         &\Delta T_{i k}+D_X T_{i k} \notag\\
         &=-2 \sum_{j, l=1}^{n-1} R_{i j k l} T^{j l}-2 R \operatorname{Ric}_{i k}\notag\\
         &+2|\operatorname{Ric}(g)|^2\left(g_{i k}-\partial_i f \partial_k f\right)+O\left(r^{-\frac{5}{2}-16\delta}\right)\notag.
      \end{align}
 Hence,
 \begin{align}
         &\Delta\left(|T|^2\right)+\left\langle X, \nabla\left(|T|^2\right)\right\rangle \notag\\
         &=2|D T|^2-4 \sum_{j, l=1}^{n-1} R_{i j k l} T^{i k} T^{j l}-4 R \sum_{i, k=1}^n \operatorname{Ric}_{i k} T^{i k} \notag\\
         &+4|\operatorname{Ric}(g)|^2 \sum_{i, k=1}^n\left(g_{i k}-\partial_i f \partial_k f\right) T^{i k}
         +O\left(r^{-\frac{5}{2}-16\delta}\right)|T|\notag \\
         &=2|D T|^2-4 \sum_{j, l=1}^{n-1} R_{i j k l} T^{i k} T^{j l}-\frac{4}{n-1} R|T|^2\notag \\
         &+4\left(|\operatorname{Ric}(g)|^2-\frac{1}{n-1} R^2\right) \sum_{i, k=1}^n\left(g_{i k}-\partial_i f \partial_k f\right) T^{i k}+O\left(r^{-\frac{5}{2}-16\delta}\right)|T|.\notag
\end{align}
 Note
 \begin{align}
    |X|^2= |\nabla f|^2=1-R=1+O(r^{-1}).\notag
 \end{align}
Then
 \begin{align}
        & \sum_{i, k=1}^n (g_{i k} -\partial_i f \partial_k f) T^{i k} \notag\\
        &= -R+ 2R|\nabla f|^2-R|\nabla f|^4-(n-1)\operatorname{Ric}(X,X)\notag\\
         &= -R^3-(n-1)\operatorname{Ric}(X,X)=O(r^{-2}).\notag
     \end{align}
 Therefore, by (\ref{Brendle-R-Ric}),  we get
 \begin{align}\label{equation-DT22}
            &\Delta\left(|T|^2\right)+\left\langle X, \nabla\left(|T|^2\right)\right\rangle \notag\\
         &\geq-4 \sum_{j, l=1}^{n-1} R_{i j k l} T^{i k} T^{j l}-\frac{4}{n-1} R|T|^2\notag\\
         &-O\left(r^{-\frac{5}{2}-16\delta}\right)|T|-O\left(r^{-\frac{9}{2}}\right).
\end{align}

 On the other hand,
  by (\ref{Brendle-fRm}),  we have
 \begin{align}\label{AC1}
            R_{i j k l} &=\frac{1}{(n-1)(n-2)} R\left(g_{i k}-\partial_i f \partial_k f\right)\left(g_{j l}-\partial_j f \partial_l f\right) \notag\\
         &-\frac{1}{(n-1)(n-2)} R\left(g_{i l}-\partial_i f \partial_l f\right)\left(g_{j k}-\partial_j f \partial_k f\right)\notag \\
         &+O\left(r^{-1-8\delta}\right).
\end{align}
Moreover,  we can estimate
 \begin{align}
     &\operatorname{tr}(T)=-R^2=O\left(r^{-2}\right), \notag\\
     &T(\nabla f, \cdot)=(n-1) \operatorname{Ric}(\nabla f, \cdot)-R^2 \nabla f=O\left(r^{-\frac{3}{2}}\right), \notag\\
     &T(\nabla f, \nabla f)=(n-1) \operatorname{Ric}(\nabla f, \nabla f)-R^2|\nabla f|^2=O\left(r^{-2}\right).\notag
 \end{align}
 Thus
 \begin{align}
             &\sum_{j, l=1}^{n-1} R_{i j k l} T^{i k} T^{j l}\notag \\
         &=\frac{1}{(n-1)(n-2)}R(-R^3-(n-1)\operatorname{Ric}(X,X))\left(g_{j l}-\partial_j f \partial_l f\right)T^{jl} \notag\\
         &-\frac{1}{(n-1)(n-2)} R|T|^2+\frac{1}{(n-1)(n-2)} R((n-1)\operatorname{Ric}(X,X)-R+R|\nabla f|^4)^2 \notag\\
         &+\frac{1}{(n-1)(n-2)}R((n-1)(\operatorname{Ric}(X,e_j)\partial_lf+\operatorname{Ric}(X,e_l)\partial_jf)-2R^2\partial_jf\partial_lf)T^{jl} \notag\\
         &-\frac{1}{(n-1)(n-2)}R((n-1)\operatorname{Ric}(X,X)\partial_jf\partial_lf-R^2|\nabla f|^2\partial_jf\partial_lf)T^{jl} \notag\\
         &+o\left(r^{-1}\right)|T|^2\notag\\
         &=-\frac{1}{(n-1)(n-2)} R|T|^2+O\left(r^{-\frac{5}{2}-8\delta}\right)|T|+o\left(r^{-1}\right)|T|^2+O(r^{-5}).\notag
   \end{align}
 Hence, inserting the above relation
into (\ref{equation-DT22}),  we obtain
 \begin{align}\label{equation-T23}
             &\Delta\left(|T|^2\right)+\left\langle X, \nabla\left(|T|^2\right)\right\rangle\notag \\
         &\geq-\frac{4(n-3)}{(n-1)(n-2)} R|T|^2\notag\\
         &-o\left(r^{-1}\right)|T|^2-O\left(r^{-\frac{5}{2}-8\delta}\right)|T|-O\left(r^{-4-8\delta}\right)
\end{align}

Let  $\Sigma$ be the level surface of $\{f=r\}$ on $(M, g)$.   Since  the mean curvature $H$ with respect to the metric $g$   on $\Sigma$  satisfies
\[H=\frac{1}{|\nabla f|}R-\frac{1}{|\nabla f|^3}\operatorname{Ric}(X,X)=O(r^{-1}), \]
by Lemma \ref{DT},  we can rewrite  (\ref{equation-T23}) as
\begin{align}
        &\Delta_{\Sigma}\left(|T|^2\right)+\left\langle X, \nabla\left(|T|^2\right)\right\rangle \notag\\
        &=\Delta\left(|T|^2\right)+\left\langle X, \nabla\left(|T|^2\right)\right\rangle-D_{X,X}^2(|T|^2)-H\langle X, \nabla |T|^2 \rangle\notag\\
        &\geq-\frac{2(n-3)}{n-2} f^{-1}|T|^2
        -o\left(r^{-1}\right)|T|^2-O\left(r^{-\frac{5}{2}-8\delta}\right)|T|-O\left(r^{-4-8\delta}\right),\notag
\end{align}
where $ \Delta_{\Sigma}$ denotes the Laplacian operator  with respect to the induced metric on   $\Sigma$.
Furthermore,  by the fact $f=const.$ on $\Sigma$,   we can estimate
\begin{align}
          &\Delta_{\Sigma}\left(f^2|T|^2\right)+\left\langle X, \nabla\left(f^2|T|^2\right)\right\rangle\notag \\
        &=f^2\Delta_{\Sigma}\left(|T|^2\right)+f^2\left\langle X, \nabla\left(|T|^2\right)\right\rangle+2f|\nabla f|^2|T|^2\notag \\
        &\geq-\frac{2(n-3)}{n-2} f|T|^2+2f(1-O(r))|T|^2-o(r)|T|^2-O\left(r^{-\frac{1}{2}-8\delta}\right)|T|-O\left(r^{-2-8\delta}\right)\notag\\
        &\geq \frac{2}{n-2} f|T|^2-o(r)|T|^2-O\left(r^{-\frac{1}{2}-8\delta}\right)|T|-O\left(r^{-2-8\delta}\right) \notag\\
        &\geq-O\left(r^{-2-8\delta}\right).\notag
\end{align}
Note that  $f^2|T|^2\rightarrow 0$ as $x$  goes to the  infinity.  Hence by  the parabolic maximum principle,  we conclude
$$f^2|T|^2\leq O(r^{-1-8\delta}),$$
which   implies (\ref{T-norm}).

 \end{proof}

 By Proposition \ref{Improve-T},
 we see
 $$|T|\leq O(r^{-\frac{3}{2}-4\delta}).$$
   By  Shi's estimates for  higher order derivatives of curvature \cite{Shi},  it follows
     \begin{align}\label{Shi-est}|D^mT|\leq O(r^{\frac{m+2}{2}}),  ~\forall   m\ge 1.
     \end{align}
     Thus  by the standard interpolation inequalities,
     we  get
      \begin{align}\label{derivative-T}|DT|\leq O(r^{-2-2\delta})~ {\rm and}~ |D^2T|\leq O(r^{-\frac{5}{2}-\delta}).
      \end{align}
       On the other hand,  by the relation
 \begin{align*}
     \begin{aligned}
         D^iD^k T_{i k} &=\frac{n-3}{2} \Delta R+D_XD_XR+\frac{7}{2}RD_XR+R^3\\
         & +\partial_i R \operatorname{Ric}_i^k \partial_k f+R|\operatorname{Ric}(g)|^2 \\
         &=\frac{n-3}{2} \Delta R+D_XD_XR+O\left(r^{-\frac{5}{2}-16\delta}\right)\\
         &=\frac{n-3}{2} \Delta R-D_X(\Delta R+2|\operatorname{Ric}(g)|^2)+O\left(r^{-\frac{5}{2}-16\delta}\right)\\
         &=\frac{n-3}{2} \Delta R-\Delta (D_XR)-R_{mii}^l\partial_lR \partial^m f+O\left(r^{-\frac{5}{2}-16\delta}\right)\\
        \end{aligned}
\end{align*}
and
$$D_XR=-\Delta R+ |{\rm Ric}(g)|^2,$$
we know from  ({\ref{derivative-f}),
  $$ D^iD^k T_{i k} =\frac{n-3}{2} \Delta R+O\left(r^{-\frac{5}{2}-16\delta}\right).$$
Hence  by (\ref{derivative-T}), we  improve the Lapalace  estimate of scalar curvature in (\ref{Brendle-R}) by
 \begin{align}\label{Delta-R}
     |\Delta R|\leq O(r^{-\frac{5}{2}-\delta}).
 \end{align}

By (\ref{Delta-R}),   (\ref{Brendle-fR}) in Lemma  \ref{Brendle-round-2}  can be improved as follows.

 \begin{prop}\label{Improve-fR}
 \begin{align}\label{Improve-fR-equation} fR=\frac{n-1}{2}+O(r^{-\frac{1}{2}-\delta}).
 \end{align}

 \end{prop}

 \begin{proof}

 By Proposition \ref{Improve-T},  we have
 \[|\operatorname{Ric}(g)|=\frac{1}{n-1} R|g-d f \otimes d f|+O\left(r^{-\frac{3}{2}-4\delta}\right)=\frac{1}{\sqrt{n-1}} R+O\left(r^{-\frac{3}{2}-4\delta}\right),\]
 and so
 \[|\operatorname{Ric(g)}|^2=\frac{1}{n-1}R^2+O(r^{-\frac{5}{2}-4\delta}).\]
 By (\ref{Delta-R}),   it follows
 \[-\langle X, \nabla R\rangle=\Delta R+2|\operatorname{Ric}(g)|^2=\frac{2}{n-1} R^2+O\left(r^{-\frac{5}{2}-\delta}\right).\]
 Thus
 \[\left\langle X, \nabla\left(\frac{1}{R}-\frac{2}{n-1} f\right)\right\rangle=O\left(r^{-\frac{1}{2}-\delta}\right).\]
 Integrating the  above relation along the integral curves of $X$,  we obtain
 \[\frac{1}{R}-\frac{2}{n-1}f=O\left(r^{\frac{1}{2}-\delta}\right).\]
 This implies (\ref{Improve-fR-equation}).

\end{proof}

Also we prove

 \begin{lem}\label{Improve-D2R}
\begin{align}\label{second-derivatives} & i). ~ |\nabla _{e_i}e_j|\leq O(r^{-\frac{1}{2}-8\delta});\notag\\
 & ii).  ~ |D_{i,j}^2R|\leq O(r^{-\frac{5}{2}-\delta});\notag\\
& iii).  ~|D_{i,j}^2\operatorname{Ric}(g)|\leq O(r^{-\frac{5}{2}-\delta}).
\end{align}

 \end{lem}

 \begin{proof}   i).  By (\ref{Brendle-Ric}) and (\ref{derivative-f}), we have
 \begin{align}\label{DX-connection}
           D_X\nabla _{e_i}e_j &= D_iD_j\nabla f+R_{kij}^l\partial^kf\notag\\
         &= D_i\operatorname{Ric}(e_j,\cdot)+R_{kij}^l\partial^kf\notag\\
         &= O(r^{-\frac{3}{2}-8\delta}).
      \end{align}
 Then
  $$|D_X\nabla _{e_i}e_j|\leq O(r^{-\frac{3}{2}-8\delta}). $$
   Since $|\nabla _{e_i}e_j|\rightarrow 0$ at infinity,
    by  integrating (\ref{DX-connection}) along the integral curves of $X$,  we get  the first relation in (\ref{second-derivatives}) immediately.

 ii).  By definition,   we have
 \begin{align}\label{two-derivative} D_{i,j}^2R=D_iD_jR-D_{\nabla _{e_i}e_j}R.
 \end{align}
 Since
 \[\begin{aligned}
     D^kT_{ik} &= \frac{n-3}{2} \partial_i R+\langle \nabla f,\nabla R\rangle\partial_i f+R^2 \partial_i f+R\operatorname{Ric}_i^k\partial_k f\\
     &=\frac{n-3}{2}\partial_i R+ (-\Delta R+ |{\rm Ric}(g)|^2)\partial_i f +O(r^{-2})\notag\\
    &=\frac{n-3}{2}\partial_i R+O(r^{-2}),
\end{aligned}\]
 By (\ref{derivative-T}),
  we get
   \begin{align}\label{derivative-R-2}|DR|=O(r^{-2}),
   \end{align}
   which is also an improvement of  the first relation  in  (\ref{Brendle-R}).
    Thus  by  the proved relation in  i),  we obtain
 \begin{align}\label{D-connection-R}
     |D_{\nabla _{e_i}e_j}R|\leq O(r^{-\frac{5}{2}-8\delta}).
\end{align}

 On the other hand,   we also have
 \begin{align}\label{D1T}
        & D^k T_{i k}\notag\\
         &= \frac{n-3}{2} \partial_i R -\Delta R\partial_if- 2|\operatorname{Ric}(g)|^2\partial_if +R^2 \partial_i f+R \operatorname{Ric}_i^k \partial_k f.
    \end{align}
Then
 \begin{align}\label{DDT}
        & D_jD^k T_{i k} \notag\\
        &=\frac{n-3}{2} \partial_j\partial_i R-\partial_j\Delta R\partial_if -\Delta R\operatorname{Ric}_{ij} -4D_j\operatorname{Ric}_{kl}\operatorname{Ric}^{kl}\partial_if \notag\\
        &-2|\operatorname{Ric}(g)|^2\operatorname{Ric}_{ij}
         +2R\partial_jR \partial_i f +R^2\operatorname{Ric}_{ij} +\partial_jR \operatorname{Ric}_i^k \partial_k f \notag\\
         &+R\partial_j\operatorname{Ric}_i^k \partial_k f+R\operatorname{Ric}_i^k\operatorname{Ric}_{jk}.
    \end{align}
 Thus  by (\ref{derivative-T}), we get
  $$|D_iD_jR|\leq O(r^{-\frac{5}{2}-\delta}).$$
Hence, inserting this relation together with (\ref{D-connection-R}) into (\ref{two-derivative}), we prove  the second  relation in (\ref{second-derivatives}).

iii).  By $(\ref{T})$, we have
 \[(n-1)D_{i,j}^2\operatorname{Ric}_{kl}=D_{i,j}^2(T_{kl}+R(g_{kl}-\partial_kf\partial_lf)).\]
 Note
 \begin{align}
        & D_{i,j}^2(R(g_{kl}-\partial_kf\partial_lf)) \notag\\
        &= D_iD_j(R(g_{kl}-\partial_kf\partial_lf))-D_{\nabla_{e_i}e_j}(R(g_{kl}-\partial_kf\partial_lf))\notag\\
         &= D_{i,j}^2R(g_{kl}-\partial_kf\partial_lf)-\partial_jR\operatorname{Ric}_{ik}\partial_lf-\partial_jR\operatorname{Ric}_{il}\partial_kf
         \notag\\
         &-\partial_iR\operatorname{Ric}_{jk}\partial_lf-\partial_iR\operatorname{Ric}_{jl}\partial_kf-R\partial_i\operatorname{Ric}_{jk}\partial_lf-R\partial_i\operatorname{Ric}_{jl}\partial_kf\notag\\
         &-R\operatorname{Ric}_{jk}\operatorname{Ric}_{il}-R\operatorname{Ric}_{jl}\operatorname{Ric}_{ik}\notag\\
         &+R\operatorname{Ric}(\nabla_{e_i}e_j,e_k)\partial_lf+R\operatorname{Ric}(\nabla_{e_i}e_j,e_l)\partial_kf.\notag
   \end{align}
   Then by the estimate proved in ii),  we see
   $$ D_{i,j}^2(R(g_{kl}-\partial_kf\partial_lf))\leq O(r^{-\frac{5}{2}-\delta}).$$
   Thus combining  (\ref{derivative-T}),  we also get the third  relation in  (\ref{second-derivatives}).

\end{proof}

At last,  we  improve (\ref{Brendle-fRm}) in Lemma \ref{Brendle-round-2} as follows.

 \begin{prop}\label{Improve-fRm}
   \begin{align}\label{Improve-fRm-equation}
    f R_{i j k l} &=\frac{1}{2(n-2)}\left(g_{i k}-\partial_i f \partial_k f\right)\left(g_{j l}-\partial_j f \partial_l f\right)\notag \\
         &-\frac{1}{2(n-2)}\left(g_{i l}-\partial_i f \partial_l f\right)\left(g_{j k}-\partial_j f \partial_k f\right)\notag\\
         &+O(r^{-\frac{1}{2}-\delta}).
\end{align}

\end{prop}

 \begin{proof}  By  \cite[Proposition 2.10]{Bre5},  we have
 \begin{equation*}
     \begin{aligned}
         -D_X R_{i j k l} &=D_{i, k}^2 \operatorname{Ric}_{j l}-D_{i, l}^2 \mathrm{Ric}_{j k}-D_{j, k}^2 \mathrm{Ric}_{i l}+D_{j, l}^2 \mathrm{Ric}_{i k} \\
         &+\sum_{m=1}^n \operatorname{Ric}_i^m R_{m j k l}+\sum_{m=1}^n \operatorname{Ric}_j^m R_{i m k l}.
 \end{aligned}
 \end{equation*}
 On the other hand,  by (\ref{Brendle-fRm}),
 \[\begin{aligned}
     \sum_{m=1}^n \operatorname{Ric}_i^m R_{m j k l} &=\frac{1}{n-1} R \sum_{m=1}^n\left(\delta_i^m-\partial_i f \partial^m f\right) R_{m j k l}+O\left(r^{-\frac{5}{2}-4\delta}\right) \\
     &=\frac{1}{n-1} R R_{i j k l}+O\left(r^{-\frac{5}{2}-4\delta}\right).
 \end{aligned}\]
 Then,   by Lemma \ref{Improve-D2R},  we get
 \[\begin{aligned}
     -D_XR_{ijkl}&=\frac{2}{n-1}RR_{ijkl}+O\left(r^{-\frac{5}{2}-4\delta}\right)\\
     &=f^{-1}R_{ijkl}+\left(r^{-\frac{5}{2}-\delta}\right).
 \end{aligned}
 \]
 It follows
 \begin{align}\label{fR}|D_X(fR_{ijkl})|\leq O(r^{-\frac{3}{2}-\delta}).
 \end{align}

 Let
 \[\begin{aligned}
     S_{i j k l} &=\frac{1}{2(n-2)}\left(g_{i k}-\partial_i f \partial_k f\right)\left(g_{j l}-\partial_j f \partial_l f\right) \\
     &-\frac{1}{2(n-2)}\left(g_{i l}-\partial_i f \partial_l f\right)\left(g_{j k}-\partial_j f \partial_k f\right).
 \end{aligned}\]
 It is easy to see
$$|D_XS_{ijkl}|\leq C|\operatorname{Ric}(X,\cdot)|. $$
Note that by (\ref{Ric2}) and (\ref{Brendle-R}) it holds
$$\operatorname{Ric}(X,\cdot)\leq O(r^{-\frac{3}{2}-8\delta}).$$
Then
\begin{align}\label{DS}|D_XS_{ijkl}|\leq O(r^{-\frac{3}{2}-8\delta}).
\end{align}
Hence,  combining (\ref{fR}) and (\ref{DS}), we obtain
 \begin{align}\label{R-S}|D_X(fR_{ijkl}-S_{ijkl})|\leq O(r^{-\frac{3}{2}-\delta}).
 \end{align}
 Since $|fR_{ijkl}-S_{ijkl}|\rightarrow 0$ at infinity,   by the  integration   along the integral curves of $X$,  we deduce from (\ref{R-S}),
 \[|fR_{ijkl}-S_{ijkl}|\leq O(r^{-\frac{1}{2}-\delta}),\]
which implies  (\ref{Improve-fRm-equation}).
\end{proof}

\subsection{Fine roundness of  level sets}

 By the asymptotically cylindrical property,
 each  level surface $\Sigma_r \subset M$  of $f$  is  diffeomorphic to $S^{n-1}$.  Then there is a family of diffeomorphisms
 $$F_r: S^{n-1}\rightarrow \Sigma_r$$
  such that $\frac{\partial}{\partial r}F_r=\frac{X}{|X|^2}$.    Thus
   $$\gamma_r=\frac{1}{2(n-2)r}F_r^*(g)$$
   defines a family of metrics on  $S^{n-1}$.

By  Proposition \ref{Improve-fRm},  we prove

 \begin{prop}\label{metric-level} For each $l\geq 0$,  it holds
\begin{align}\label{metric-level-re} \left\|\frac{d}{d r} \gamma_r\right\|_{C^l\left(S^{n-1}, \gamma_r\right)} \leq O\left(r^{-\frac{3}{2}-\frac{1}{2}\delta}\right).
\end{align}

\end{prop}

 \begin{proof}  Let
 $$\bar R=\bar R_{ijkl}dx^i\otimes dx^j \otimes dx^k\otimes dx^l$$
 be  the curvature tensor of  induced  metric  $g|_ {\Sigma_r}$ on    $\Sigma_r$.
 Then by the Gauss formula,
 \[\bar{R}_{i j k l} -{R}_{i j k l}=\frac{1}{|\nabla f|^2}(\operatorname{Ric}_{ik}\operatorname{Ric}_{jl}-\operatorname{Ric}_{il}\operatorname{Ric}_{jk}).\]
 By Proposition \ref{Improve-fRm},   it follows
 \begin{align}\label{Improve-Rm-level}
              \bar{R}_{i j k l} &=\frac{1}{2(n-2)r}\left(g_{i k}-\partial_i f \partial_k f\right)\left(g_{j l}-\partial_j f \partial_l f\right) \notag\\
         &-\frac{1}{2(n-2)r}\left(g_{i l}-\partial_i f \partial_l f\right)\left(g_{j k}-\partial_j f \partial_k f\right) \notag\\
         &+O(r^{-\frac{3}{2}-\delta}).
     \end{align}
 Since the normal velocity of the flow  is  given by
 $$\frac{1}{|X|}=1+O(r^{-1}),$$
 we have (cf. \cite[Section 3]{DZ1}),
 \[\frac{d}{d r} F_r^*(g)=F_r^*(\frac{2}{|\nabla f|^2}\operatorname{Ric(g)}).\]
 Hence  by (\ref{Improve-Rm-level}),  we get
 \[\sup _{S^{n-1}}\left|\frac{d}{d r} F_r^*(g)-\frac{1}{r} F_r^*(g)\right|_{F_r^*(g)} \leq O\left(r^{-\frac{3}{2}-\delta}\right).\]
 Note
 \[\frac{d}{d r} \gamma_r=\frac{1}{r}(\frac{d}{d r} F_r^*(g)-\frac{1}{r} F_r^*(g)).\]
Therefore, we obtain
 \begin{align}\label{metric-diff-level}
     \sup _{S^{n-1}}\left|\frac{d}{d r} \gamma_r\right|_{\gamma_r} \leq O\left(r^{-\frac{3}{2}-\delta}\right).
 \end{align}

 By  (\ref{Improve-Rm-level}) we see  that the manifold $(S^{n-1},\gamma_r)$ has uniformly  bounded curvature and all the derivatives of the curvature are also uniformly  bounded.  Thus
  \[\|\frac{1}{r} \gamma_r\|_{C^l\left(S^{n-1}, \gamma_r\right)}\leq O({r^{-1}}).\]
 On the other hand,  by  Shi's  estimate
  $$\sup _{\{f=r\}}\left|D^l \operatorname{Ric}\right| \leq O\left(r^{-\frac{l+2}{2}}\right), $$
which is equal to
 \[\sup _{\{r-\sqrt{r} \leq f \leq r+\sqrt{r}\}}\left|D^l\left(\mathscr{L}_{\frac{X}{|X|^2}}(g)\right)\right| \leq O\left(r^{-\frac{l+2}{2}}\right),\]
 we have
 \[\left\|F_r^*\left(\mathscr{L}_{\frac{X}{|X|^2}}(g)\right)\right\|_{C^l\left(S^{n-1}, \gamma_r\right)} \leq O(1) .\]
 Since
 \[\frac{d}{d r} \gamma_r+\frac{1}{r} \gamma_r=\frac{1}{2(n-2)r} F_r^*\left(\mathscr{L}_{\frac{X}{|X|^2}}(g)\right),\]
 we  get
 \begin{align}\label{metric-diff-sphere}
     \left\|\frac{d}{d r} \gamma_r\right\|_{C^l\left(S^{n-1}, \gamma_r\right)} \leq O\left(r^{-1}\right)
 \end{align}
 for each $l\geq 0$.  Now we can obtain   (\ref{metric-level-re}) by  $(\ref{metric-diff-level})$ and $(\ref{metric-diff-sphere})$ via the standard interpolation inequalities .

\end{proof}

 Since the metrics $\gamma_r$ converge in $C^\infty$ to a smooth metric $\bar{\gamma}$ as $r\rightarrow \infty$.  By $(\ref{Improve-Rm-level})$,   the limit metric $\bar{\gamma}$ must be the standard sphere metric with constant sectional curvature $1$. Moreover,  by Proposition \ref{metric-level},  we get
 \begin{align}\label{diff-metric}
     \left\|\gamma_r-\bar{\gamma}\right\|_{C^l\left(S^{n-1}, \bar{\gamma}\right)} \leq O\left(r^{-\frac{1}{2}-\frac{1}{2}\delta}\right),
 ~{\rm for ~ each}~ l\geq0.
 \end{align}

 \section{Solutions  of  (\ref{vector-equ}) }

 In this section,  we  construct a collection of  VFs  $W_{a}$, $a \in \{1,...,\frac{n(n-1)}{2}\}$,  which  solve (\ref{vector-equ}) with property
 \begin{align}\label{W-property}|W_a|= O(r^{\frac{1}{2}})~{\rm and}~ |DW_{a}|\leq O(r^{-\frac{1}{2}}).
 \end{align}

As in \cite{Bre2},  we  let $\bar{U}_{a}, a \in \{1,...,\frac{n(n-1)}{2}\}$,  be the Killing  VFs on $(S^{n-1},\bar{\gamma})$ such that
 \begin{align}\label{VF-sphere}
     \sum_{\alpha=1,...,\frac{n(n-1)}{2}} \bar{U}_{a} \otimes \bar{U}_{a}=\frac{1}{2}\sum_{i=1}^{n-1}\bar{e}_i\otimes\bar{e}_i,
 \end{align}
 where $\{\bar{e}_i,...,\bar{e}_{n-1}\}$ is a local orthonormal frame on $(S^{n-1},\bar{\gamma})$.
 Then for the each  $\bar{U_{a}}$ above,  we define a VF    $U_{a}$ on $M$ such that   outside a compact set  of $M$ it holds
 \begin{align}\label{ua-vf}F_r^*U_{a}=\bar{U_{a}},
 \end{align}
  which is tangent  to the level sets  $\Sigma_r$ for all $r$ sufficiently large.
 By $(\ref{diff-metric})$,  we see
  \begin{align}\label{Ua-vf}\sum_{a=1, ..., \frac{n(n-1)}{2}}  U_{a} \otimes U_{a}=r\left(\sum_{i=1}^{n-1} e_i \otimes e_i+O\left(r^{-\frac{1}{2}-\frac{1}{4}\delta}\right)\right),
 \end{align}
 where $\{e_1,...,e_{n-1}\}$ is a local orthonormal frame on  ${\Sigma_r}$.
 Moreover,
 \begin{align}\label{Lie-D-metric-sphere}
     \left\|\mathscr{L}_{\bar{U}_a}\left(\gamma_r\right)\right\|_{C^l\left(S^{n-1}, \gamma_r\right)}=\left\|\mathscr{L}_{\bar{U}_a}\left(\gamma_r-\bar{\gamma}\right)\right\|_{C^l\left(S^{n-1}, \gamma_r\right)} \leq O\left(r^{-\frac{1}{2}-\frac{1}{2}\delta}\right).
 \end{align}

 \begin{prop}\label{App-Killing-VF}   Each  $U_{a}$ above satisfies $(\ref{u-vector})$  on $(M,g)$. Namely,
 \begin{align}\label{h-norm-2}\sup _{\Sigma_r} \left|  \mathscr{L}_{U_{\alpha}}(g) \right| \leq O\left(r^{-\frac{1}{2}-\frac{1}{2}\delta}\right)
 \end{align}
and
 \begin{align}\label{Q-norm}
     \sup _{\Sigma_r}\left|\Delta U_a+D_X U_a\right| \leq O\left(r^{-1-\frac{1}{4}\delta}\right).
 \end{align}
\end{prop}

 \begin{proof}
  By
  \begin{align}\label{orthogonal-X}[U_{a},\frac{X}{|X|^2}]=0,
 \end{align}
  we have
  \begin{align}\label{Lie-B-UX}
         \left[U_a, X\right]=U_a\left(|X|^2\right) \frac{X}{|X|^2}.
  \end{align}
  Since by  $(\ref{Brendle-R})$,
  $$ U_{a}(|X|^2)=-\langle U_{a},\nabla R\rangle=O(r^{-1-8\delta}), $$
  we get
  \begin{align}\label{norm-U}|\left[U_a, X\right]|\leq O(r^{-1-8\delta}).
  \end{align}

  By $(\ref{Lie-D-metric-sphere})$ and $(\ref{norm-U})$,  it is easy to see
  \begin{align}
  &\mathscr{L}_{U_{\alpha}}(g)(e_i,e_j)\notag\\
  &=\left\langle D_{e_i} U_a, e_j\right\rangle+\left\langle D_{e_j} U_a, e_i\right\rangle=O\left(r^{-\frac{1}{2}-\frac{1}{2}\delta}\right),\notag
  \end{align}
 \begin{align}
 & \frac{1}{2}\mathscr{L}_{U_{\alpha}}(g)(X,X)\notag\\
 &=\left\langle D_X U_a, X\right\rangle=\left\langle D_{U_a} X, X\right\rangle-\left\langle\left[U_a, X\right], X\right\rangle=-\frac{1}{2} U_a\left(|X|^2\right)=O\left(r^{-1-8\delta}\right).\notag
  \end{align}
  Moreover,
 \begin{align}
 & \mathscr{L}_{U_{\alpha}}(g)(X,e_j)\notag\\
 &=\left\langle D_X U_a, e_j\right\rangle+\left\langle D_{e_j} U_a, X\right\rangle=\left\langle D_{U_a} X, e_j\right\rangle-\left\langle U_a, D_{e_j} X\right\rangle-\left\langle\left[U_a, X\right], e_j\right\rangle=0.\notag
   \end{align}
 Thus we  prove (\ref{h-norm-2}) by the above three  relations.

  Since
  $$\left\|\bar{U}_a\right\|_{C^l\left(S^{n-1}, \gamma_r\right)} \leq O(1),$$
    we have
  $$\sup _{\Sigma_r}\left|D_{\Sigma}^l U_a\right| \leq O\left(r^{-\frac{l-1}{2}}\right), ~\forall~l\geq 1.$$
  Together with   (\ref{orthogonal-X}),  we derive
  \begin{align}\label{deriv-Ua}\sup _{\Sigma_r}  \left|D^l U_a\right| \leq O\left(r^{-\frac{l-1}{2}}\right),~\forall~l\geq 1.
  \end{align}
 As a consequence,
   $$\sup _{\Sigma_r}\left|D^lh_a\right| \leq O(r^{-\frac{l}{2}}), ~\forall~ l\geq 1,$$
   where $h_a= \mathscr{L}_{U_{\alpha}}(g).$
   Hence by  the standard interpolation inequalities,  we obtain
  \begin{align}\label{derivative}\sup _{\{f=r\}}\left|Dh_a\right| \leq O(r^{-1-\frac{1}{4}\delta}).
  \end{align}
  On the other hand,  we have
  \[\operatorname{div}\left(h_a\right)-\frac{1}{2} \nabla\left(\operatorname{tr} (h_a)\right)=\Delta U_a+\operatorname{Ric}\left(U_a, \cdot \right). \]
  Therefore,
  $$\sup _{\Sigma_r} (\Delta U_a+\operatorname{Ric}(U_a, \cdot))  \leq O\left(r^{-1-\frac{1}{4}\delta}\right).
  $$
  Since
   $$\left|\operatorname{Ric}\left(U_a, \cdot \right)-D_X U_a\right|=\left|\left[U_a, X\right]\right|,$$
   by (\ref{norm-U}),
 $$  \left|\operatorname{Ric}\left(U_a, \cdot \right)-D_X U_a\right|\leq O\left(r^{-1-8\delta}\right).$$
 Combining the above two relations, we  prove  (\ref{Q-norm}) immediately.

\end{proof}

By  Proposition \ref{App-Killing-VF}, there are $\frac{n(n-1)}{2}$ VFs  $U$ which satisfy  the elliptic equation,
\begin{align}\label{V-equation}\Delta U+D_XU=Q,
\end{align}
where $Q$ satisfies
 \begin{align}
\label{Q-decay}|Q|= O\left(r^{-1-\frac{1}{4}\delta}\right).
\end{align}
We need to perturb each  $U$ to get a solution of (\ref{vector-equ}). First we have the following  decay estimate.

 \begin{lem}\label{MP-VF} Let $U$  be a solution of $(\ref{V-equation})$ in the region $\{f\leq \rho\}$,  which satisfies
  $(\ref{Q-decay})$.  Then there exist a uniform constant $B\geq 1$ such that
 \begin{align}\label{MP-VF-ineq}
     \sup_{\{f\leq \rho\}}(|U|-Bf^{-\epsilon})\leq \sup_{f=\rho}|U|-B\rho^{-\epsilon}.
 \end{align}

\end{lem}

 \begin{proof}
  By the identity (\ref{scalar-equ}),  we have
  \begin{align}\label{f-epsilon-ineq}
         &-\Delta(f^{-\epsilon})-\langle X,\nabla(f^{-\epsilon})\rangle\notag\\
         &=\epsilon f^{-1-\epsilon}(\Delta f+|\nabla f|^2)-\epsilon(\epsilon+1)f^{-2-\epsilon}|\nabla f|^2\notag\\
         &>\frac{\epsilon}{2}f^{-1-\epsilon}.
      \end{align}
Then by  (\ref{Q-decay}) with $\epsilon=\frac{1}{4}\delta$,   there is a constant $B>1$ such that
 \[|Q|<\frac{\epsilon B}{2}f^{-1-\epsilon}.\]
   On the other hand,  by the  Kato's inequality,
 \begin{equation*}
     \begin{aligned}
         \Delta (|U|^2)+\langle X,\nabla (|U|^2)\rangle&=2|DU|^2+2\langle U,Q\rangle\\
         &\geq 2|\nabla|U||^2-2|Q||U|,
     \end{aligned}
 \end{equation*}
 we have
 \begin{align}\label{V-norm-ineq}
     \Delta (|U|)+\langle X,\nabla |U|\rangle\geq -|Q|.
 \end{align}
Thus we get
 \[\Delta (|U|-Bf^{-\epsilon})+\langle X,\nabla(|U|-Bf^{-\epsilon})\rangle>0.\]
  Hence, by the maximum principle, we prove (\ref{MP-VF-ineq}).

\end{proof}

By Lemma \ref{MP-VF}, we are able to prove

 \begin{theo}\label{V-estimate} Let $U$ be a smooth VF on $(M, g)$  which satisfies $(\ref{V-equation})$ and   $(\ref{Q-decay})$.  Then there exists a smooth bounded  solution  $V$ of  $(\ref{vector-equ})$   on $(M, g)$,  which satisfies
 \begin{align}\label{gradient-V}|DV|\leq O(r^{-\frac{1}{2}}).
 \end{align}

 \end{theo}

 \begin{proof} Let  ${\rho_m}$ be a sequence   to the infinity and $V^{(m)}$  a sequence of solutions  of  (\ref{V-equation})  in the region $\{f\leq \rho\}$ with zero boundary  values on the boundary $\Sigma_{\rho_m}$. Then  by
Lemma \ref{MP-VF}, we see
 \begin{align}\label{decay-U}\sup_{\{f\le {\rho_m}\}}|V^{(m)}|\leq B2^{-\epsilon}-B\rho_m^{-\epsilon}\leq C,
 \end{align}
 where $C$ is independent of $m$.    By the regularity,  we  also get
$$ \sup_{\{f\le {\rho_m}\}}  |D^kV^{(m)}|\le C_k.$$
Thus by taking a subsequence of  $V^{(m)}$, we will obtain a smooth bounded  solution  $V$ of  (\ref{vector-equ})   on $(M, g)$

It remains  to prove  (\ref{gradient-V}).  Let  $r_m\rightarrow 0$ be  a sequence and
 \[\hat{g}^{(m)}(t)=r_m^{-1}\Phi^*_{r_mt}(g),\]
 where $t\in[-\frac{1}{2},0]$.   We define
 \[\hat{V}^{(m)}(t)=\Phi_{r_m t}^*(V)\]
 and
 \[\hat{Q}^{(m)}(t)=r_m \Phi_{r_m t}^*(Q).\]
 Then   $\hat{V}^{(m)}(t)$ satisfies  the parabolic equation,
 \[\frac{\partial}{\partial t} \hat{V}^{(m)}(t)=\Delta_{\hat{g}^{(m)}(t)} \hat{V}^{(m)}(t)+\operatorname{Ric}_{\hat{g}^{(m)}(t)}\left(\hat{V}^{(m)}(t)\right)-\hat{Q}^{(m)}(t) .\]
 By   (\ref{Q-decay}),  we see
 \[\sup_{t\in[-\frac{1}{2},0]}\sup_{\{r_m-\sqrt r_m\leq f\leq r_m+\sqrt r_m\}}\left|\hat{Q}^{(m)}(t)\right|_{\hat{g}^{(m)}(t)} \leq O\left(r_m^{-\frac{1}{2}- \epsilon}\right)\]
 and
 \[\sup_{t\in[-\frac{1}{2},0]}\sup_{\{r_m-\sqrt r_m\leq f\leq r_m+\sqrt r_m\}}\left|\hat{V}^{(m)}(t)\right|_{\hat{g}^{(m)}(t)} \leq O\left(r_m^{-\frac{1}{2}}\right).\]
 Thus by the standard interior estimates for parabolic equations (cf. \cite[Theorem 7.22]{Lib}),  we obtain
 \begin{align}
    & \sup _{\left\{f=r_m\right\}}\left|D \hat{V}^{(m)}(0)\right|_{\hat{g}^{(m)}(0)}\notag\\
    &\leq C\sup_{t\in[-\frac{1}{2},0]}\sup_{\{r_m-\sqrt r_m\leq f\leq r_m+\sqrt r_m\}}\left|\hat{V}^{(m)}(t)\right|_{\hat{g}^{(m)}(t)}\notag\\
     &+C\sup_{t\in[-\frac{1}{2},0]}\sup_{\{r_m-\sqrt r_m\leq f\leq r_m+\sqrt r_m\}}\left|\hat{Q}^{(m)}(t)\right|_{\hat{g}^{(m)}(t)}\notag\\
     &\leq O(r_m^{-\frac{1}{2}}).\notag
 \end{align}
 Hence by  taking a limit,   we prove (\ref{gradient-V}).

\end{proof}

 \begin{cor}\label{VF-W} There exists a collection of  VFs $W_{a}$, $a \in \{1,...,\frac{n(n-1)}{2}\}$, each of  which solves  $(\ref{vector-equ})$,
$$\Delta W_{a}+D_XW_{a}=0,~ {\rm on}~ M.$$
  Moreover,  $W_{a}$ satisfies   $(\ref{W-property})$.

 \end{cor}

 \begin{proof} Let $Q=\Delta U_a+D_X U_a$ for each $U_a$ constructed in (\ref{ua-vf}).
  Then by  Lemma \ref{V-estimate},  we obtain a solution  $V_{a}$ of (\ref{V-equation}).  Thus  $W_{a}=U_{a}-V_{a} $ solves (\ref{vector-equ}) which satisfies (\ref{W-property}) by  (\ref{Ua-vf}),  (\ref{deriv-Ua}) and  Theorem \ref{V-estimate}.

 \end{proof}

 \begin{rem}\label{remark-brendle-1}
 In \cite[Section 3]{Bre2},  Brendle  used  a delicate blow-down analysis for the equation $(\ref{V-equation})$ to get  a locally uniformly bounded estimate for the sequence of  $V^{(m)}$.    We  can prove that $V^{(m)}$ are uniformly bounded  by the decay estimate in Lemma  $(\ref{MP-VF})$ (also see $(\ref{decay-U})$),   so in particular  simple his proof to get a limit solution of  $(\ref{vector-equ})$.

\end{rem}

 \section{Almost Killing VFs with fast decay of  their Lie derivatives }

 In this section, we prove that there exists a collection of approximate Killing VFs $Y_{a}$, $a \in \{1,...,\frac{n(n-1)}{2}\}$, each of which has fast decay of the  Lie derivative for the  soliton metric $(M, g)$ near the  infinity.

 Recall the  Lichnerowicz Laplacian $\Delta_L$  for  symmetric  2-tensors associated to $(M, g)$  defined by
 \begin{align}\label{LL}\left(\Delta_L h\right)_{i j} \doteqdot \Delta h_{i j}+2 R_{ikjl} h^{k l}-\operatorname{Ric}_{i k} h^k_{ j}-\operatorname{Ric}_{j k} h^k_{i}.
 \end{align}
 Then the Ricci curvature tensor is a special solution of (\ref{L-equ}) (cf. \cite{Bes, Bre1, Bre2}).  Namely, we have

 \begin{align}\label{ricci-equation}\Delta_L(\operatorname{Ric(g)})+\mathscr{L}_X(\operatorname{Ric}(g))=0.
 \end{align}

 In \cite{Bre2},  Brendle found a  way to construct the solution of (\ref{L-equ})  by solving  the  elliptic  equation  (\ref{vector-equ}) for VFs. In fact, if $W$ is a solution of   (\ref{vector-equ}),  then $h=\mathscr{L}_Wg$ will satisfy the  Lichnerowicz-type  equation  (\ref{L-equ}), i.e.,
  \[\Delta_L h+\mathscr{L}_Xh =0.\]

The following lemma shows that the  solution of (\ref{L-equ}) has any polynomial decay near  the infinity  associated  to an asymptotically cylindrical steady Ricci soliton.

 \begin{lem} \label{MP-tensor} Let $(M, g)$  be an asymptotically cylindrical steady Ricci soliton.
  Let $h$  be a solution  of the Lichnerowicz-type equation $(\ref{L-equ})$  on the region $\{\rho_0'\leq f\leq \rho\}$ associated to $(M, g)$.   Then for any large $l$, there exists a $\rho_0\ge \rho_0'$
 such that
\begin{align}\label{decay-tube} \sup _{\{\rho_0\leq f \leq \rho\}} f^l|h| \leq \max\{ \rho^l \sup _{\{f=\rho\}}|h|,\rho_0^l \sup _{\{f=\rho_0\}}|h| \}.
\end{align}
 \end{lem}

 \begin{proof}  Since $(M,g)$ is asymptotically cylindrical,  there exists a set $K\subset M $ such that $(M,g)$ admits positive curvature operator outside of $K$  by Lemma \ref{DZ-6.6}.  Moreover,   by  (\ref{Brendle-fRic}) in Lemma \ref{Brendle-round},  for any  large $l>0$  there is another  compact set $K'$ such that
 \begin{align}\label{l-K'-set} f\operatorname{Ric}<(\frac{l}{2}-\frac{l(l+1)}{2}f^{-1}|\nabla f|^2)g,    ~M\setminus K'.
 \end{align}
 Thus we can  a $\rho_0$ large enough such that $K\cup K'\subset\{f\leq\rho_0\}$.

 It suffices to show that
 \begin{align}\label{MP-tensor-equation}
     f^lh\leq \max\{ \rho^l \sup _{\{f=\rho\}}|h|,\rho_0^l \sup _{\{f=\rho_0\}}|h|\}g.
 \end{align}
 Then  $-h$ will  be also  true since $-h$ is also a solution of (\ref{L-equ}).  Thus  $(\ref{MP-tensor-equation})$  implies (\ref{decay-tube}).

Following the argument in \cite[Section 4]{Bre2},  we choose the smallest real number $\theta$ such that $\theta f^{-l}g-h$ is positive semi-definite at each point in the region $\{\rho_0\leq f\leq \rho\}$. Then there exists a point $p_0\in\{\rho_0\leq f\leq \rho\}$ and an othonormal basis $\{e_1,...,e_n\}$ of $T_{p_0}M$ such that  at the point $p_0$ it holds
 \begin{align}\label{minimal-p}\theta f^{-l}-h(e_1,e_1)=0.
 \end{align}
 Without loss of generality,  we may assume  that $p_0\in\{\rho_0< f< \rho\}$. Thus we get at $p_0$,
\begin{align}\label{maximal-principle} \theta\Delta(f^{-l})-(\Delta h)(e_1,e_1)\geq0
\end{align}
 and
 \begin{align}\label{minimal-p-2}\theta\langle X,\nabla(f^{-l})\rangle-(D_Xh)(e_1,e_1)=0.
 \end{align}

Note
$$\Delta(f^{-l})=l(l+1)f^{-l-2}|\nabla f|^2-lf^{-l-1}\Delta f$$
and
 $$\langle X,\nabla(f^{-l})\rangle=-lf^{-l-1}|\nabla f|^2.$$
 Then  by  (\ref{L-equ}) with help of   (\ref{minimal-p})-(\ref{minimal-p-2}),    we have at $p_0$,
 \[\begin{aligned}
 0 &=(\Delta h)\left(e_1, e_1\right)+\left(D_X h\right)\left(e_1, e_1\right)+2 \sum_{i, k=1}^n R\left(e_1, e_i, e_1, e_k\right) h\left(e_i, e_k\right) \\
 & \leq \theta \Delta\left(f^{-l}\right)+\theta\left\langle X, \nabla\left(f^{-l}\right)\right\rangle+2 \sum_{i, k=1}^n R\left(e_1, e_i, e_1, e_k\right) h\left(e_i, e_k\right) \\
 &=-l \theta f^{-l-1}\left(1-(l+1) f^{-1}|\nabla f|^2-\frac{2}{l}f \operatorname{Ric}\left(e_1, e_1\right)\right) \\
 &-2 \sum_{i, k=1}^n R\left(e_1, e_i, e_1, e_k\right)\left(\theta f^{-l} g\left(e_i, e_k\right)-h\left(e_i, e_k\right)\right).
 \end{aligned}\]
  Since $(M,g)$ has positive sectional curvature on $M\setminus K$,
 \[\sum_{i, k=1}^n R\left(e_1, e_i, e_1, e_k\right)\left(\theta f^{-l} g\left(e_i, e_k\right)-h\left(e_i, e_k\right)\right)\geq 0.\]
 Thus we obtain
 $$-l \theta f^{-l-1}\left(1-(l+1) f^{-1}|\nabla f|^2-\frac{2}{l}f \operatorname{Ric}\left(e_1, e_1\right)\right)\geq0, ~{\rm on}~M\setminus K\cup K'.$$
By (\ref{l-K'-set}),  it follows  that
 $\theta\leq0$.   Since $h\leq\theta f^{-l}g$,  we conclude  that $h\leq0$ at each point in the region $\{\rho_0\leq f\leq \rho\}$.
 Hence,  $(\ref{MP-tensor-equation})$ is automatically   satisfied.  The lemma is proved.

 \end{proof}

The following  lemma  is due to \cite[Lemma 4.1]{Bre2}.

 \begin{lem}\label{Brendle-Lich}  Let  $(S^{n-1}\times \mathbb R, \bar{g}(t))$  $(t\in(0,1))$  be   the shrinking cylinders,   namely,    $\bar{g}(t)$  is a form of metric on $S^{n-1}\times \mathbb R$,
 \begin{align}\label{shrinking-cyl}
     \bar{g}(t)=(n-2)(2-2 t) g_{S^{n-1}}+dr^2.
 \end{align}
 Suppose that  $\bar{h}(t)$  $(t\in(0,1))$ is a  solution of  $(\ref{para-L})$ associated to  $\bar{g}(t)$, which
 is  invariant under translations along the axis of the cylinder, and satisfies
 \begin{align}\label{bounded condition}
     |\bar{h}(t)|_{\bar{g}(t)} \leq(1-t)^{-l}
 \end{align}
 for some  $l>0$ and all $t\in (0,\frac{1}{2}]$. Then
 $$\inf _{\lambda \in \mathbb{R}} \sup _{S^{n-1} \times \mathbb{R}}\left|\bar{h}(t)-\lambda \operatorname{Ric}({\bar{g}(t)})\right|_{\bar{g}(t)} \leq N(1-t)^{\frac{1}{2(n-2)}-\frac{1}{2}}, ~\forall~ t\in [\frac{1}{2},1).$$
  Here $N$ is a positive constant.

\end{lem}

By  Lemma \ref{MP-tensor}   and Lemma \ref{Brendle-Lich}, we prove the following decay estimate.

 \begin{theo}\label{VF-decay} For all large $l$, there exists a collection of appoximate Killing VFs  $Y_a$, $a \in \{1,...,\frac{n(n-1)}{2}\}$, such that
 \begin{align}\label{Lie-decay}
     |\mathscr{L}_{Y_a}g| \leq O(r^{-l}),
 \end{align}
 outside a compact set $K$ of $M$.  Moreover,  $\mathscr{L}_{Y_a}g$  are solutions of Lichnerowicz type equation $(\ref{L-equ})$ on $M$.

\end{theo}

 \begin{proof}  Let $W_a$, $a \in \{1,...,\frac{n(n-1)}{2}\}$,  be a   a collection of VFs  on $M$  constructed in Corollary \ref{VF-W}. Then    $h_a=\mathscr{L}_{W_a}g$ satisfy  (\ref{L-equ}).  Moreover,  by (\ref{h-norm-2}),
  \begin{align}\label{ha-decay}|h_a|\leq O(r^{-\frac{1}{2}}).
  \end{align}

   As in \cite[Section 4]{Bre2},  we define a  function on $[r_0, \infty)$ by
 \[A_a(r)=\inf _{\lambda \in \mathbb{R}} \sup _{\{f=r\}} | h_a-\lambda \operatorname{Ric(g)} | .\]
 Clearly
 \begin{align}\label{Brendle-Ar}
     A_a(r)\leq \sup _{\{f=r\}}|h_a|\leq O(r^{-\frac{1}{2}})\leq O(r^{\frac{1}{2(n-2)}-\frac{1}{2}-\epsilon'}),
 \end{align}
 where  $\epsilon'=\frac{1}{1000n}$ is small number.   Fix a  number $\tau\in(0,\frac{1}{2})$ such that $\tau^{-\epsilon'}>2N$, where $N$ is the constant in Lemma \ref{Brendle-Lich}. Then by $(\ref{Brendle-Ar})$ there exists a sequence of  $r_m\rightarrow \infty$ such that
 \[A_a\left(r_m\right) \leq 2 \tau^{\frac{1}{2}-\frac{1}{2(n-2)}+\epsilon'} A_a\left(\tau r_m\right).\]

 For the above sequence $\{r_m\}$,   there are    two cases as  follows.

\textbf{Case 1}:  There exists a subsequence of $r_m\rightarrow\infty$, we still denote them by $r_m$, such that
 \begin{align}\label{case1}
     \rho_0^l\sup_{\{f=\rho_o\}}|h_a-\lambda_m\operatorname{Ric(g)}|\leq r_m^l\sup_{\{f=r_m\}}|h_a-\lambda_m\operatorname{Ric(g)}|.
 \end{align}
 We claim that there is some $\lambda\in\mathbb{R}$ such that
  \begin{align}\label{identity-case}
  h_a=\lambda\operatorname{Ric(g)}, ~ {\rm on}~ \{f\geq \rho_0\}.
  \end{align}
  Then   we let $Y_a=W_a-\frac{1}{2}\lambda X$, which   satisfies
  $$\mathscr{L}_{Y_a}g=0, ~\{f\geq \rho_0\}, $$
 and  so  (\ref{Lie-decay}) is obviously  true.   By Corollary \ref{VF-W} and  the fact $|X|\leq 1$,  we notice that  $Y_a$ is a non-zero VF   and it is globally  well-defined on $M$.

 To prove  (\ref{identity-case}),  we need to show  that there exists a sequence of  numbers $r_m\rightarrow \infty$ such that $A_a(r_m)=0$.    Namely,
 \begin{align}\label{zero-condition}h_a-\lambda_m\operatorname{Ric(g)}=0,   ~{\rm on}~\{f=r_m\}.
 \end{align}
In fact,     by  applying  Lemma \ref{MP-tensor}  to $h_a-\lambda_m\operatorname{Ric(g)} $  on the region $\{\rho_0\leq f\leq r_m\}$,  it follows by (\ref{case1}),
  we have
 \[\begin{aligned}
    & \sup _{\{\rho_0\leq f \leq r_m\}} f^l|h_a-\lambda_m\operatorname{Ric(g)}| \\
    &\leq \max\{ r_m^l \sup _{\{f=r_m\}}|h_a-\lambda_m\operatorname{Ric(g)}|,\rho_0^l \sup _{\{f=\rho_0\}}|h_a-\lambda_m\operatorname{Ric(g)}|\}\\
     &\leq r_m^l \sup _{\{f=r_m\}}|h_a-\lambda_m\operatorname{Ric(g)}|=0.
 \end{aligned}\]
 Thus $ h_a-\lambda_m\operatorname{Ric(g)}=0$ in the region $\{\rho_0\leq f\leq r_m\} $ for any $\lambda_m$.   As a consequence,   all $\lambda_m$ are  same and so (\ref{identity-case}) is true.

Suppose that   (\ref{zero-condition}) is not true.  Then there is a  sufficiently large $r$ such that
 \begin{align}\label{Ar-nq0} A_a(r)>0.
 \end{align}
  By choosing  a $\lambda_m$ such  that $r \in [\rho_0,r_m]$ and then   applying Lemma \ref{MP-tensor} to the tensor
 $$\tilde{h}_a^{(m)}=\frac{1}{A_a\left(r_m\right)}\left(h_a-\lambda_m \operatorname{ Ric(g) }\right),$$
   it follows by  (\ref{case1}),
 \begin{align}\label{normalize-h}
    & \sup _{\{f=r\}} |\tilde{h}_a^{(m)}| \notag\\
    & \leq \frac{r_m^l}{r^l} \sup _{\{f=r_m\}} |\tilde{h}_a^{(m)}|=\frac{r_m^l}{r^l A_a(r_m)} \sup _{\{f=r_m\}}|h_a-\lambda_m \operatorname{Ric(g)}| \notag\\
    &=\frac{r_m^l}{r^l}.
 \end{align}

Define
 \[\hat{g}^{(m)}(t)=r_m^{-1} \Phi_{r_m t}^*(g)\]
 and
 \[\hat{h}_a^{(m)}(t)=r_m^{-1} \Phi_{r_m t}^*\left(\tilde{h}_a^{(m)}\right).\]
 Then the  metrics $\hat{g}^{(m)}(t)$ evolve by the Ricci flow for each $m$  and the tensors $\hat{h}_a^{(m)}(t)$ satisfy the parabolic Lichnerowicz equation associated to the flow $\hat{g}^{(m)}(t)$,
 \[\frac{\partial}{\partial t} \hat{h}_a^{(m)}(t)=\Delta_{L, \hat{g}^{(m)}(t)} \hat{h}_a^{(m)}(t).\]
 By $(\ref{normalize-h})$, it is easy to see
 \[\limsup_{m\rightarrow\infty}\sup_{t\in[\delta',1-\delta']}\sup_{\{r_m-\delta'^{-1}\sqrt{r_m}\leq f\leq r_m+\delta'^{-1}\sqrt{r_m}\}}\left|\hat{h}_a^{(m)}(t)\right|_{\hat{g}^{(m)}(t)}<\infty, \]
 where   $ \delta'\in(0,\frac{1}{2})$ is any fixed number.

 Choose a sequence of marked points $p_m\in M$ such that  $f(p_m)=r_m$.
 Then  manifolds $(M,\hat{g}^{(m)}(t),$ $~p_m)$ converge  to a  family of shrinking cylinders $(S^{n-1}\times \mathbb{R}, \bar{g}(t))$, $t\in(0,1)$ in the  sense  of  Cheeger-Gromov,  where $\bar{g}(t)$ is  given by $(\ref{shrinking-cyl})$.  Moreover,  the vector fields $r_m^{\frac{1}{2}}X$ converge to the axial vector field $\frac{\partial}{\partial z}$ on $S^{n-1}\times \mathbb{R}$, and
  the sequence $\hat{h}_a^{(m)}(t)$ converges to a  family of tensors $\bar{h}_a(t)$, $t\in(0,1)$, which solve the parabolic Lichnerowicz equation  (\ref{para-L})  associated to  $\bar{g}(t)$.

 Note that
 \[\Phi_{\sqrt{r_m} s}^*\left(\hat{h}_a^{(m)}(t)\right)=\hat{h}_a^{(m)}\left(t+\frac{s}{\sqrt{r_m}}\right).\]
 Then
  $$\Psi_s^*(\bar{h}_a(t))=\bar{h}_a(t), $$
    where $ \Psi_s:S^{n-1}\times \mathbb{R}\rightarrow S^{n-1}\times \mathbb{R}$ denotes the flow generated by the axial vector field $-\frac{\partial}{\partial z}$.  Namely,  $\bar{h}_a(t)$ is invariant under translations along the axis of the cylinder.  Moreover, by $(\ref{normalize-h})$, we  have
 $$|\bar{h}_a(t)|_{\bar{g}(t)} \leq (1-t)^{-l}, ~\forall ~t\in (0,\frac{1}{2}].  $$
  Thus  by Lemma \ref{Brendle-Lich},  we get
 \begin{align}\label{normalize-h-Ric1}
     \inf _{\lambda \in \mathbb{R}} \sup _{S^{n-1} \times \mathbb{R}}\left|\bar{h}_a(t)-\lambda \operatorname{Ric}({\bar{g}(t)})\right|_{\bar{g}(t)} \leq N (1-t)^{\frac{1}{2(n-2)}-\frac{1}{2}}, ~\forall ~t\in[\frac{1}{2},1).
 \end{align}
  On the other hand,  we have
 \[\begin{aligned}
     &\inf _{\lambda \in \mathbb{R}} \sup _{\Phi_{r_m(\tau-1)}\left(\left\{f=\tau r_m\right\}\right)}\left|\hat{h}_a^{(m)}(1-\tau)-\lambda \operatorname{Ric}({\hat{g}^{(m)}(1-\tau))}\right|_{\hat{g}^{(m)}(1-\tau)}\\
     &=\inf _{\lambda \in \mathbb{R}} \sup _{\left\{f=\tau r_m\right\}}\left|\tilde{h}_a^{(m)}-\lambda \operatorname{Ric}(g)\right|_g\\
     &=\frac{1}{A_a\left(r_m\right)} \inf _{\lambda \in \mathbb{R}} \sup _{\left\{f=\tau r_m\right\}}\left|h_a-\lambda \operatorname{Ric}(g)\right|_g\\
     &=\frac{A_a(\tau r_m)}{A_a(r_m)}\\
     &\geq \frac{1}{2} \tau^{\frac{1}{2(n-2)}-\frac{1}{2}-\epsilon'}.
 \end{aligned}\]
 Then  as  $m\rightarrow \infty$ we obtain
 \begin{align}\label{normalize-h-Ric2}
     \inf_{\lambda\in \mathbb{R}}\sup_{S^{n-1}\times \mathbb{R}}|\bar{h}_a(1-\tau)-\lambda\operatorname{Ric}({\bar{g}(1-\tau)})|_{\bar{g}(1-\tau)}\geq \frac{1}{2} \tau^{\frac{1}{2(n-2)}-\frac{1}{2}-\epsilon'}.
 \end{align}
Hence,  by  choosing a small $\tau$ so that  $\tau^{-\epsilon'}>2N$,   we see that $(\ref{normalize-h-Ric1})$  is a contradiction with  $(\ref{normalize-h-Ric2})$.    Therefore    we prove  (\ref{identity-case}).

 \textbf{Case 2}:  There is a $\rho_1$ large enough, such that
 \begin{align}\label{case2}
     \rho_0^l\sup_{\{f=\rho_o\}}|h_a-\lambda_m\operatorname{Ric}(g)|> r_m^l\sup_{\{f=r_m\}}|h_a-\lambda_m\operatorname{Ric}(g)|,
 \end{align}
 for all $r_m>\rho_1$.   Then  applying   Lemma \ref{MP-tensor} to $h_a-\lambda_m\operatorname{Ric(g)} $ on region $\{\rho_0\leq f\leq r_m\}$, for all $r_m>\rho_1$, we  have
 \begin{align}\label{h-Ric-decay}
         &\sup _{\{f=r\}} f^l|h_a-\lambda_m\operatorname{Ric}(g)| \notag\\
          &\leq \max\{ r_m^l \sup _{\{f=r_m\}}|h_a-\lambda_m\operatorname{Ric}(g)|,\rho_0^l \sup _{\{f=\rho_0\}}|h_a-\lambda_m\operatorname{Ric}(g)|\}\notag\\
         &\leq \rho_0^l \sup _{\{f=\rho_0\}}|h_a-\lambda_m\operatorname{Ric}(g)|, ~r\in[\rho_0,r_m].
 \end{align}

 \begin{claim}\label{claim-1}  The sequence $\lambda_m$ are uniformly  bounded.
 \end{claim}

 By the above claim,   we can take a subsequence of $\lambda_m$  which converges to a bounded  number $\lambda$, and the VF $W_a-\frac{1}{2}\lambda_m X$ converges to a vector field $Y_a=W_a-\frac{1}{2}\lambda X$ on $M$ smoothly.  Thus by  $(\ref{h-Ric-decay})$,  we obtain
 \begin{align}\label{Y-decay}
     \sup _{\{f=r\}}|h_a-\lambda\operatorname{Ric(g)}|\leq\frac{\rho_0^l }{r^l} \sup _{\{f=\rho_0\}}|h_a-\lambda\operatorname{Ric}(g)|\leq O(r^{-l}).
 \end{align}
 Hence,   (\ref{Lie-decay}) holds for the above $Y_a$. The proof of Theorem \ref{VF-decay} is finished.

  On the contrary,    there exists a subsequence $\lambda_m\rightarrow \infty$ such that  $\lambda_{m+1}-\lambda_m>1$.    Note that  $\lambda_m$ is chosen by the relation
  $$A_a(r_m)=\sup _{\{f=r_m\}} | h_a-\lambda_m \operatorname{Ric(g)} |=\inf_{\lambda}\sup _{\{f=r_m\}} | h_a-\lambda \operatorname{Ric(g)} |.$$
  Since $|h_a|\leq O(r^{-\frac{1}{2}})$ by (\ref{ha-decay}) and   $|\operatorname{Ric(g)}|\geq O(r^{-1})$ by (\ref{Ric1}),
  it  is easy to see that
 \begin{align}\label{lambda-m-growth}
     \lambda_m\leq O(r_m^{\frac{1}{2}}).
 \end{align}

 By $(\ref{h-Ric-decay})$, we have
 \begin{equation*}
     \sup _{\{f=r\}} |h_a-\lambda_m\operatorname{Ric(g)}|\leq \frac{C'}{r^l}\lambda_m, ~r\in[\rho_0,r_m].
 \end{equation*}
 In particular,
 \begin{align}\label{lambda-m-m}
     \sup _{\{f=r_m\}} |h_a-\lambda_m\operatorname{Ric(g)}|\leq \frac{C'}{r_m^l}\lambda_m.
 \end{align}
It follows that
  \begin{align}\label{lambda-m-m+1}
     \sup _{\{f=r_{m}\}} |h_a-\lambda_{m+1}\operatorname{Ric(g)}|\leq \frac{C'}{r_m^l}\lambda_{m+1}.
 \end{align}
 Note that
 \begin{align*}
     &\sup _{\{f=r_{m}\}}(\lambda_{m+1}-\lambda_m)|\operatorname{Ric(g)}|\\
     &\leq\sup _{\{f=r_{m}\}} |h_a-\lambda_{m+1}\operatorname{Ric(g)}|+\sup _{\{f=r_{m}\}} |h_a-\lambda_{m}\operatorname{Ric(g)}|.
 \end{align*}
 Thus combining  (\ref{lambda-m-growth})-(\ref{lambda-m-m+1}),   we get
 \begin{align}\label{diff-lambda}
     \begin{aligned}
         &(\lambda_{m+1}-\lambda_m)\frac{c'}{r_m}\\
         &\leq\sup _{\{f=r_{m}\}}(\lambda_{m+1}-\lambda_m)|\operatorname{Ric}(g)|\leq \frac{C'}{r_m^l}(\lambda_m+\lambda_{m+1})\\
         &\leq\frac{C'}{r_m^l}(\lambda_{m+1}-\lambda_m)+\frac{C''}{r_m^{l-\frac{1}{2}}}.
     \end{aligned}
 \end{align}
As a consequence,   by choosing  $l\ge 2$,  we derive
 \begin{equation*}
     (\lambda_{m+1}-\lambda_m)\leq O(r_m^{-l+\frac{3}{2}}),
      \end{equation*}
 and so by (\ref{lambda-m-growth}),
 \begin{align}\label{lambda-m+1-growth}
     \lambda_{m+1}\leq O(r_m^{\frac{1}{2}}).
 \end{align}
 Therefore,  by (\ref{diff-lambda}), we conclude
 \begin{align}\label{decay-contra}
     &\sup _{\{f=r_{m}\}}(\lambda_{m+1}-\lambda_m)|\operatorname{Ric(g)}|\leq\frac{C'}{r_m^l}(\lambda_m+\lambda_{m+1})\notag\\
     &\leq O(r_m^{-l+\frac{1}{2}}).
 \end{align}
However, by the fact $\lambda_{m+1}-\lambda_m>1$, we see
$$ \sup _{\{f=r_{m}\}}(\lambda_{m+1}-\lambda_m)|\operatorname{Ric(g)}|
\ge O(r_m^{-1}),$$
which is a contradiction with (\ref{decay-contra})!   Claim \ref{claim-1} is prove.

\end{proof}

\begin{rem}\label{remark-brendle-2}Because of lack of  the positivity  of sectional curvature of $g$, we only do the decay estimates for  the tensors $h_a$ and $\mathscr{L}_{Y_a}g$  away from a compact set of $M$  in Lemma  \ref{MP-tensor} and Theorem \ref{VF-decay}, respectively.  Moreover, we need to do the fast decay estimate for those tensors  in oder to use the heat kernel method in the   latter  Section 5.

\end{rem}

 \section{Evolution solution of  Lichnerowicz  equation}
  In this section, we study the parabolic   equation  of  Lichnerowicz   (\ref{para-L})  by using the heat kernel method  under the curvature-pinching condition (\ref{pinching-condition}).

 Let  $g(t)$  ($t\in (0, T]$)  be  a solution  of Ricci flow and $G$  the heat kernel of following  heat-type equation associated to  $g(t)$,
 \begin{align}\label{H-equation}
     \frac{\partial}{\partial t}H=\Delta_{g(t)} H+2PH,
 \end{align}
 where  $P=P(\cdot, t)$ is  a smooth function on $ M \times  (0, T]$.    Namely, $G>0$ satisfies
 \begin{align}\label{heat-kernel}
 \partial_t G(x, t ; y, s) &=\Delta_{x, t} G(x, t ; y, s)+2P G(x, t ; y, s)\notag \\
 \lim _{t \searrow s} G(\cdot, t ; y, s) &=\delta_y.
 \end{align}
 The existence of heat kernel $G$
  on a  complete  manifold can be found in \cite[Chapter 24]{CC}.  Moreover,  $G$ satisfies the following estimate,
 \begin{align}\label{HK-upper-bound}
     G(x, t ; y, 0) \leq \frac{C_3e^{C_1t} \exp \left(-\frac{d_{g(t)}^2(x, y)}{C_4t}\right)}{\operatorname{Vol}_{g(t)}^{1 / 2} B_{g(t)}\left(x, \sqrt{\frac{t}{2}}\right) \cdot \operatorname{Vol}_{g(t)}^{1 / 2} B_{g(t)}\left(y, \sqrt{\frac{t}{2}}\right)}.
 \end{align}

 We first prove the existence of  positive  bounded  super-solution of  (\ref{H-equation}).

 \begin{lem}\label{supersolution} Let  $(M^n, g)$ $(n\geq 4)$  be  a steady gradient Ricci soliton  and $g(t)$   the soliton  Ricci flow of $g$.  Suppose that $g$ has  an asymptotically cylindrical property and satisfies the following curvature-pinching condition,
  \begin{align}\label{m-strong}2P<\frac{m}{m+1}R,
  \end{align}
  where $P$ is the  function defined by (\ref{P-defi}).
  Then  for any large  $m\ge 2$ there exists a smooth positive uniformly bounded function $s(x,t)$ such that
 \begin{align}\label{s-ineq}
     (\frac{\partial}{\partial t}-\Delta_{g(t)}-2P(\cdot, t))s(x,t)\geq 0,
 \end{align}
 where $P(\cdot, t)=\Phi^*_t P(\cdot)$.
 Moreover $s(x,0)$ satisfies
 \begin{align}\label{decay-s}
\lim_{\rho(x) \to \infty}f(x)^{m} s(x,0)=1,
 \end{align}
  where $\rho(x)=d(p,x)$ and $f$ is the defining function of  $(M^n, g)$.

\end{lem}

 \begin{proof} We reduce to find a smooth positive uniformly bounded function $s$  on $M$ such that
 \begin{align}\label{elliptic-s}
     \Delta s+D_X s+2P s\leq 0.
 \end{align}
 Then  $s(x,t)=\Phi^*_t(s)$  will satisfy $(\ref{s-ineq})$, where  $\Phi^*_t$ is the transformations group generated by $-X$.

  By a direct  calculation,
 \begin{align}\label{ef-ineq}
        & \Delta e^{-f}+ D_X e^{-f}+ 2P e^{-f}\notag\\
        &=-\Delta f e^{-f}+|\nabla f|^2 e^{-f}-|\nabla f|^2 e^{-f}+2P e^{-f}\notag\\
         &=-(R-2P) e^{-f}\leq -\frac{1}{m+1}Re^{-f}.
 \end{align}
Then
 \begin{align}\label{fm-ineq}
        & \Delta f^{-m}+ D_X f^{-m}+ 2P f^{-m}\notag\\
        &= -mf^{-m-1}\Delta f+ m(m+1)f^{-m-2}|\nabla f|^2\notag\\
        &-mf^{-m-1}|\nabla f|^2+2P f^{-m}\notag\\
         &\leq-mf^{-m-1}(1-(m+1)|\nabla f|^2f^{-1}-\frac{1}{m+1}Rf).
    \end{align}
 Note that $f\geq 4$ and by (\ref{Brendle-fR}) we may choose $m$ large enough, such that $\frac{1}{m+1}\frac{n}{2}<\frac{1}{4} $.  Thus there exists  a compact set $K_1$ such that $\frac{1}{m+1}Rf <\frac{1}{4}$ outside of $K_1$.
  Moreover,  we can choose another compact set $K_2$ such that $(m+1)f^{-1}<\frac{1}{4}$ outside of $K_2$. Let $K'=K_1\cup K_2$.  Hence,  by $(\ref{fm-ineq})$ we get
 \begin{align}\label{fm-outside}
     \Delta f^{-m}+ D_X f^{-m}+ 2P f^{-m}\leq 0, ~M\setminus K'.
 \end{align}

 By the  compactness of $K'$, it is easy to see that on $K'$ it holds
 \begin{align}\label{fm-bound-K}
     |mf^{-m-1}(1-(m+1)|\nabla f|^2f^{-1}-\frac{1}{m+1}Rf)|\leq C_1,
 \end{align}
 and
 \begin{align}\label{ef-bound-K}
     \frac{1}{m+1}Re^{-f}>C_2,
 \end{align}
 where $C_1$ and $C_2$ are two constants.
 Let
 \begin{align}\label{s-def}
     s(\cdot)=\frac{C_1+1}{C_2}e^{-f}+f^{-m}.
 \end{align}
 Then by (\ref{ef-ineq})  and (\ref{fm-outside}),
  we get
   $$ \Delta s+D_X s+2P s\leq 0, ~ M\setminus K'.$$
    Moreover,  by $(\ref{ef-ineq})$, $(\ref{fm-ineq})$, $(\ref{fm-bound-K})$ and $(\ref{ef-bound-K})$,  we also obtain
    $$ \Delta s+D_X s+2P s\leq 0, ~ K'. $$
    Hence, the function $s(x)$ satisfies (\ref{elliptic-s}).
    Moreover, by $(\ref{s-def})$, $s(x)$ satisfies (\ref{decay-s}).
    The lemma is proved.

\end{proof}

Let $G(x, t; y, s)$ be the heat kernel in (\ref{heat-kernel}) for  the soliton Ricci flow $g(t)$ and  the function $P(\cdot, t)=\Phi^*_t P(\cdot)$ \footnote{We may assume that the function  $P(\cdot)$ is smooth, see Section 6 below.} in Lemma \ref{supersolution}.
For any fixed $D>0$,  we set
 \begin{align}\label{g-function}u_D(x,t)=\int_{B_0(p,D)}G(x,t; y,0)dv_{g(0)}y,
 \end{align}
  where  $p$ is  the  maximal point of $R$ and  $B_0(p,D)$ denotes the $D$-geodesic ball  centered at $p$  with respect to  the Ricci soliton $g$.  We  have the following  estimate for  $u_D(x,t)$.

 \begin{lem}\label{HK-vanish-local} Let $(M,g(t))$ be the soliton  Ricci flow  of $g$  in Lemma \ref{supersolution}.   Suppose that   $(\ref{pinching-condition})$ holds for $g$.
  Then
 \begin{align}\label{limit-u}
 \sup_{B_t(p,D)}u_D(\cdot,t)\rightarrow 0, ~{\rm as}~t\rightarrow \infty,
 \end{align}
 where $B_t(p,D)$ denotes the $D$-geodesic ball of $g(t)$, which exhausts $M$ as $t\to\infty$.

\end{lem}

 \begin{proof}
   By the definition of $G$,  we have
 \[s(x,0)=\int_MG(x,0;y,0)s(y,0)d_0y.\]
 Since  $s(x,t)$ is a super-solution of $(\ref{H-equation})$ by  Lemma \ref{supersolution}, by the  maximum principle
  we see
 \begin{align}\label{s-bound}
     s(x,t)\geq\int_MG(x,t;y,0)s(y,0)d_0y.
 \end{align}
 Note that there is some $c>0$ such that
  $$s(y,0)\geq c, ~{\rm in}~ B_0(p,D).$$
  Thus we obtain
 \begin{align}
 &c\int_{B_0(p,D)}G(x,t;y,0) dv_{g(0)}y \notag\\
 &\leq c\int_MG(x,t;y,0) dv_{g(0)}y \leq \int_MG(x,t;y,0)s(y,0)  dv_{g(0)}y\notag\\
 &\leq s(x,t)\leq C.\notag
 \end{align}
 It follows
 \begin{align}\label{u-bounded}
 u_D(x,t)\leq c^{-1}C=C_1.
 \end{align}
  Since  $u_D(x,t)$ satisfies the equation $(\ref{H-equation})$,   by a standard parabolic estimate,  we obtain
\begin{align}\label{gradient-estimate}
  |\partial_tu_D(x,t)|+|\nabla u_D(x,t)|\leq C_2.
  \end{align}

 Now we use the contradiction argument as in \cite[Section 7]{Yi2} to prove (\ref{limit-u}). Suppose that
 there exist $\epsilon>0$ and  sequences of $t_i\rightarrow \infty$ and $x_i\in B_{t_i}(p,D)$ such that $u_D(x_i,t_i)>\epsilon$. Without loss of generality we may assume that $t_{i+1}\geq t_i+1$. Thus by (\ref{gradient-estimate}),  we have
 \begin{align}\label{uD-L1-bound}
    & \int_{B_t(p,D)}\int_{B_0(p,D)}G(x.t;y.0) dv_{g(0)}y dv_{g(t)}x \notag\\
    &=\int_{B_t(p,D)}u_D(x,t) dv_{g(t)}x
    \geq \delta_1>0, ~\forall ~t\in[t_i,t_{i}+\delta_1],
 \end{align}
 where $\delta_1$ is a small constant.

 By the  Gaussian upper bound of $G$ in  $(\ref{HK-upper-bound})$,  we use  the divergence theorem to see
 $$\int_M\Delta_{x,t}G(x,t;y,0) dv_{g(t)}x =0.$$
 It follows
  \begin{align}\label{dt-uD-L1}
         &\partial_t\int_MG(x,t;y,0)   dv_{g(t)}x\notag\\
         & =\int_M [\partial_tG(x,t;y,0)-R(x,t)G(x,t;y,0)]  dv_{g(t)}x \notag\\
         &=\int_M [2P(x,t)-R(x,t)] G(x,t;y,0)  dv_{g(t)}x,
    \end{align}
 where the function $P(\cdot, t)=\Phi^*_t P(\cdot)$.
We notice that  by  (\ref{pinching-condition})  there is a $\delta_2>0$ such that
 \begin{align}\label{2P-R}
     \sup_{B_t(p,D)} [2P(\cdot, t)- R(\cdot, t)]=\sup_{B_0(p,D)} [2P(\cdot)- R(\cdot)] \leq -\delta_2.
 \end{align}

Let
$$G_D(t)=\int_{B_0(p,D)}\int_MG(x,t;y,0)  dv_{g(t)}x dv_{g(0)}y.$$
Then  by  (\ref{dt-uD-L1}),  we get
 \begin{align}\label{dt-GD}
     &\partial_t G_D(t)\notag\\
     &=\int_{B_0(p,D)}\int_M [2P(x,t)-R(x,t)] G(x,t;y,0)   dv_{g(t)}x dv_{g(0)}y <0.
 \end{align}
 More precisely,  by $(\ref{uD-L1-bound})$ and $(\ref{2P-R})$
 we obtain
  $$ \partial_t G_D(t)\leq-\delta_1\delta_2,~ t\in[t_i,t_{i}+\delta_1].$$
   Since  $G$ is everywhere positive and  so as  $G_D(t)>0$,  we  derive
 \begin{align}-G_D(0) &\leq \lim_{t\rightarrow \infty}G_D(t)-G_D(0)\notag\\
& \leq -\sum_{i=1}^\infty\int_{t_i}^{t_i+\delta_1}\delta_1\delta_2dt=-
 \sum_{i=1}^\infty\delta_1^2\delta_2.\notag
 \end{align}
The above relation  implies  that $G_D(0)=\infty$, which is impossible!.  Therefore  (\ref{limit-u}) is true and  the lemma is proved.

\end{proof}

Inspired by  a   pinching estimate  of  Anderson-Chow  for the solution of  (\ref{para-L})  associated to the $3$-dimensional  Ricci flow  \cite{AC},   we have the following  pinching estimate   for the solution of  (\ref{para-L})  in higher dimensions.

 \begin{lem}\label{h-H-pinching} Let $(M^n, g(t))$ $(t\in (0,T])$  be an $n\geq 3$-dimension Ricci flow.   Let  $H$ be  a positive solution   of $(\ref{H-equation})$ associated to the flow $g(t)$ with the function $P(\cdot, t)$ given by
 \begin{align}\label{pinching-condition-t}
  P(\cdot, t)= \sup_h\frac{{\rm Rm}_{g(t)}(\cdot)(h,h)}{|h|^2}=\sup_h\frac{{\rm R}_{t, ijkl}(\cdot)h^{ik}h^{jl}}{h_{ik}h^{ik}}, ~\forall ~t\in ~(0, T].
  \end{align}
 Then any   solution  $h$ of $(\ref{para-L})$
satisfies  the following   pinching estimate,
 \begin{align}\label{h-H-pinching-ineq}
     \partial_t(\frac{|h|^2}{H^2})\leq\Delta(\frac{|h|^2}{H^2})+\frac{2}{H}\nabla H\cdot\nabla(\frac{|h|^2}{H^2}),
 \end{align}
 where $\Delta=\Delta_{g(t)}$ and $|\cdot|=|\cdot|_{g(t)}$.

\end{lem}

 \begin{proof} By a direct computation with help of  (\ref{para-L})  and $(\ref{H-equation})$,  we have the following identities
 \[\begin{aligned}
     &\partial_t(\frac{|h|^2}{H^2})\\
     &=-\frac{2}{H^2}\frac{\partial}{\partial t}g_{ij}h^i_{k}h^{kj}+\frac{2}{H^2}(\frac{\partial}{\partial t}h)\cdot h-2\frac{|h|^2}{H^3}\frac{\partial}{\partial t}H\\
     &=\frac{4 \operatorname{Ric}(h,h)}{H^2}+\frac{2}{H^2}\left(\Delta h_{i j}+2 R_{i k j l} h^{k l}-R_{i k} h^k_{ j}-R_{j k} h^k_{ i}\right) h^{i j}-2\frac{|h|^2}{H^3}(\Delta H+2PH),
 \end{aligned}
 \]
 and
 \[\begin{aligned}
     &\Delta (\frac{|h|^2}{H^2})\\
     &=2\frac{\Delta h\cdot h}{H^2}+2\frac{|\nabla h|^2}{H^2}-8\frac{h}{H^3}\langle\nabla h,\nabla H\rangle-2\frac{|h|^2\Delta H}{H^3}+6\frac{|h|^2|\nabla H|^2}{H^4}.
 \end{aligned}
 \]
 Note
 \[\frac{1}{H}\nabla H\cdot \nabla(\frac{|h|^2}{H^2})=2\frac{h}{H^3}\langle\nabla h,\nabla H\rangle-2\frac{|h|^2|\nabla H|^2}{H^4}.
 \]
 Then combining  the above three relations,  we get
 \[\begin{aligned}
   &  \partial_t(\frac{|h|^2}{H^2})\\
   &=\Delta(\frac{|h|^2}{H^2})+\frac{2}{H}\nabla H\cdot\nabla(\frac{|h|^2}{H^2})+\frac{4}{H^2}(R_{ijkl}h^{ik}h^{jl}- P|h|^2)\\
     &-2\frac{|H\nabla_ih_{jk}-(\nabla_iH)h_{jk}|^2}{H^4}\\
     &\leq \Delta(\frac{|h|^2}{H^2})+  \frac{2}{H}\nabla H\cdot\nabla(\frac{|h|^2}{H^2}).
 \end{aligned}\]
The last inequality follows from the condition (\ref{pinching-condition-t}).

\end{proof}

 \begin{rem}\label{Anderson-Chow} 1)  In \cite{AC},  Anderson-Chow proved $(\ref{h-H-pinching-ineq})$ for  a $3$-dimensional  Ricci flow when
$H(\cdot, t)=R(g(t))$. In a recent paper \cite{Yi2},  Lai observed  that Anderson-Chow's proof  works   for any solution $H$  of
 following heat-type equation associated to  the $3$-dimensional  Ricci flow,
 \begin{align}\label{3d-H-equation}
     \frac{\partial}{\partial t}H=\Delta_{g(t)} H+2\frac{|{\rm Ric}(g(t))|^2}{R(g(t))}H.
 \end{align}
 Namely, $(\ref{h-H-pinching-ineq})$  holds  for any solution $H$  of  $(\ref{3d-H-equation})$ in the $3$-dimensional case.
In fact,   $R(g(t))$ satisfies $(\ref{3d-H-equation})$ for any dimension.

2)  It  can be checked  that $P(\cdot, t)\le \frac{|{\rm Ric}(g(t))|^2}{R(g(t))}$ in  the $3$-dimensional case (cf.  \cite{AC}).  Then the proof of  $(\ref{h-H-pinching-ineq})$  in
 Lemma \ref{h-H-pinching}  also works for the solution $H$ of $(\ref{3d-H-equation})$ as studied  by Lai above.

\end{rem}

By the above three lemmas, we are able to get an asymptotic behavior estimate for the solution of  (\ref{para-L}).

 \begin{prop}\label{h-vanish} Let $(M^n, ~ g(t))$ be the soliton  Ricci flow  of    $g$  with   dimension  $n\geq 4$ in Lemma \ref{supersolution}. Let  $h(t)$ $(t\in[0,\infty))$ be a solution of (\ref{para-L}) with an initial value $h(x, 0)$ satisfying
 \begin{align}\label{h-condition}|h|(x,0)\leq O(\rho^{-l}),
 \end{align}
 for some  $l>m$,  where   $m$ is   the large number chosen as in  $(\ref{m-strong})$.
  Then
  \begin{align}\label{h-limit}
  \lim_t\sup_{B_t(p,D)}|h|(\cdot,t)=0.
  \end{align}
\end{prop}

 \begin{proof} Let $G(x, t; y, s)$ be the heat kernel in (\ref{g-function}). Set
 $$H(x,t)=\int_MG(x,t;y,0)|h|(y,0)d_{g(0)}y.$$
Then   $H(\cdot, t)$  is a solution of $(\ref{H-equation})$  with $P(\cdot, t)=\Phi^*_t P(\cdot)$  for any $t>0$.
 Thus by Lemma \ref{h-H-pinching},  (\ref{h-H-pinching-ineq}) holds for  $\frac{|h|^2}{H^2}$.
  By the maximum principle, we see that $|h|\leq H$, that is,
 \begin{align}\label{h-G-bound}
     |h|(x,t)\leq \int_MG(x,t;y,0)|h|(y,0) d_{g(0)}y.
 \end{align}
 On the other hand,  by  (\ref{h-condition})  and (\ref{decay-s}) in Lemma \ref{supersolution},   for any $\epsilon >0$ we can find some $D>0$ such that
 $$|h|(y,0)\leq \epsilon s(y,0), ~\forall ~y\in M\setminus B_0(p,D).$$
  We may assume $D>\frac{1}{\epsilon}$.  Hence,   by   $(\ref{h-G-bound})$ and $(\ref{s-bound})$ together with   Lemma \ref{HK-vanish-local}, we see that  for sufficiently large $t$ and $x\in B_t(p,D)$,
 \[\begin{aligned}
    & |h|(x,t)\\
    &\leq \int_{B_0(p,D)}G(x,t;y,0)|h|(y,0) d_{g(0)}y+\int_{M\setminus B_0(p,D)}G(x,t;y,0)|h|(y,0) d_{g(0)}y\\
     &\leq \epsilon+\epsilon\int_MG(x,t;y,0) s(y,0) d_{g(0)}y\\
     &\leq \epsilon+\epsilon s(x,t)\leq\epsilon(1+C),
 \end{aligned}\]
 where the function $ s(x,t)$ is  uniformly bounded by Lemma \ref{supersolution}.
Since $\epsilon$ can be taken arbitrary small,  we will get  (\ref{h-limit}).
 The proposition is proved.

\end{proof}

 \section{Proof of  main theorem and corollaries}

In this section,  we complete the proofs of   Theorem \ref{main-theorem},  Corollary  \ref{Curv-decay} and  Corollary \ref{asymtotic-behavior-rigidity}.

\begin{proof} [Proof of Theorem \ref{main-theorem}]
By Lemma  \ref{DZ-6.6},  it is easy to see
$$\lim_{x\to\infty}\frac{2P(x)}{R(x)}= \frac{2}{n-1}<1$$
when $n\ge 4.$
On the other hand, by the condition (\ref{pinching-condition}),   for any  compact set $K$  there exists a  small positive number $\mu$ such that
$$2P<(1-\mu)R\le\frac{m}{m+1}R, ~{\rm on }~K,$$
where $m$ can be chosen sufficiently large.
Thus the condition (\ref{m-strong}) holds on the  total space $(M^n, g)$ when $n\ge 4.$.

 We may assume that  the function $P(x)$ is a smooth.  In fact,    by a small perturbation of $P$,  there  is always a  smooth  function $P'$  on $M$
such that
$P\le P'$ and  $P'$ still satisfies (\ref{m-strong}). Then  we replace the function $P(\cdot, t)$ in $(\ref{H-equation})$   by
$P(\cdot, t)=\Phi^*_t P'(\cdot)$  for any $t>0$ so that all of
Lemma \ref{supersolution}, Lemma \ref{HK-vanish-local} and  Proposition \ref{h-vanish} are still true.

 By Theorem \ref{VF-decay}, we already know that
 there is a collection of approximate Killing VFs $Y_a$  ($a \in \{1,...,\frac{n(n-1)}{2}\}$)  such that
 \begin{equation*}
     |\mathscr{L}_{Y_a}g|\leq O(r^{-l}),
     \end{equation*}
 outside a compact set, where $l$ can be chosen so that $l>m$.
 Let $g(t)=\Phi^*_t(g)$ and $Y_a(t)=\Phi^*_t(Y_a)$.
Then $h_a(t)=\mathscr{L}_{Y_a(t)}g(t)=\Phi^*_t  h_a$  is a solution of  (\ref{para-L})  for any $t>0$.  This is because  $h_a=\mathscr{L}_{Y_a}g$ is a solution of (\ref{L-equ}).   Thus applying  Proposition \ref{h-vanish} to each solution  $h_a(t)$  of (\ref{L-equ}) with the
initial value $h_a$,   we see
 $$ \lim_t\sup_{B_t(p,D)}|h_a|(\cdot,t)=0.$$
 Note that $h_a(\cdot)=\Phi_{-t}^* h_a(t)$.
 Hence,
 $|h_a|\equiv 0$ on any compact $D$-geodesic ball  $B_0(p,D)\subset M$.  This implies that   $|h_a|\equiv 0$ on $M$.
Therefore we prove that each  $Y_a$ is a Killing VF on $(M, g)$.

By  \cite[Proposition 4.1]{Bre2},  $Y_a$ satisfies
 $$\Delta Y_a+\operatorname{Ric}(g)(Y_a, \cdot)=0.$$
Since  $Y_a$   also satisfies  by the construction of $Y_a$,
 $$\Delta Y_a+D_X Y_a=0, $$
 we get
 $$[Y_a,X]=\operatorname{Ric}(g)(Y_a, \cdot)-D_XY_a=0.$$
  It follows
    $$ D^2(\mathscr{L}_{Y_a}(f))=\mathscr{L}_{Y_a}(D^2f)=\frac{1}{2}\mathscr{L}_{Y_a}(\mathscr{L}_{X}g)=\frac{1}{2}\mathscr{L}_{X}(\mathscr{L}_{Y_a}g)=0.
 $$
 Hence,  the function $\mathscr{L}_{Y_a}(f)=\langle Y_a,X\rangle$ is constant.  Since $X$ vanishes at the point where $f$ attains its minimum,    the function $\langle Y_a,X\rangle$  must vanish identically.   This implies that  $(M,g)$ is rotationally symmetric and  Theorem \ref{main-theorem} is proved.

\end{proof}

The first part of Corollary \ref{Curv-decay} is a direct application of Theorem \ref{main-theorem} since $(M,g)$ has an asymptotically cylindrical property under the  linear decay condition of curvature  in  (\ref{Curv-decay}) by  \cite[Theorem 1.3]{DZ1}.  Then by Proposition \ref{P-condition},
the second part is also true.

\begin{rem}\label{re-coro-0.3}In $4d$,  the condition ${\rm Rm}_g(x)\ge 0$ $(\rho(x)\ge r_0)$   in  $(\ref{Curv-decay})$  can be weakened as   the non-negativity of sectional curvature
This is because   all level sets of $f$ are  diffeomorphic   to $\mathbb S^{3}$ and so  any  limit of  Ricci flow from a   blow-down sequence of Ricci flow  obtained in   \cite[Theorem 1.4]{DZ1} satisfies   Definition \ref {def1.2}-(ii)    by Hamilton's result  for  closed $3d$-Ricci flow with positive Ricci curvature \cite{Ha82} (also see \cite[Theorem 1.5]{DZ2}).
\end{rem}

 \begin{proof}[Proof of  Corollary \ref{asymtotic-behavior-rigidity}]
We need to check that  $(M,g)$ has an asymptotically cylindrical property under   the condition
 $(\ref{asymtotic-behavior})$.   In fact,  for a fixed  $p\in (\mathbb{R}^n,g_0)$,   we  can choose   an orthonormal basis  $\{e_1,e_2,...,e_n\}$  of $T_pM$ with respect to metric $g_0$ with $e_n=\frac{\nabla f}{|\nabla f|}$.  Then   it is easy to see
 \begin{align}\label{Rm-0n}
     \sup_{\partial B_p(r)}|\hat{R}_{ijkl}-R_{ijkl}|_{g_0}\leq O(r^{-1-\tau}), ~1\leq i,j,k,l\leq n-1,
 \end{align}
  \begin{align}\label{Rm-1n}
     \sup_{\partial B_p(r)}|\hat{R}_{ijkn}-R_{ijkn}|_{g_0}\leq O(r^{-\frac{3}{2}-\tau}), ~1\leq i,j,k\leq n-1,
 \end{align}
 and
  \begin{align}\label{Rm-2n}
     \sup_{\partial B_p(r)}|\hat{R}_{inkn}-R_{inkn}|_{g_0}\leq O(r^{-2-\tau}), 1\leq i,k\leq n-1.
 \end{align}
 Here  we use the sectional curvature decay of Bryant soliton and  (\ref{derivative-f}).
By the above threes relations,  we see that $(M,g)$ has positive curvature operator outside of a compact set $K$.  Moreover,  its scalar curvature has exactly linear decay,  i.e. there exists positive numbers $r_0$, $C_1$ and $C_2$ such that
 \begin{align}\label{exact-linear}\frac{C_1}{\rho(x)}\leq R(x)\leq\frac{C_2}{\rho(x)},   ~\forall~ \rho(x)\geq r_0.
 \end{align}
Also one can check that $(M,g)$ is $\kappa$-noncollapsed for some $\kappa>0$.

According to the proof of \cite[Theorem 1.4]{DZ1},  it suffices to prove that the Gauss sphere  $\Sigma_r$ of geodesic ball of $B_r(p)$ in $(M, g)$ is diffeomorphic to  the unit sphere in the  Euclidean space and
$${\rm Diam}(\Sigma_r, g)\le C\sqrt{r}$$
as long as $r$ is sufficiently large.  But these two facts can be guaranteed by $(\ref{asymtotic-behavior})$. Thus the corollary is proved.

\end{proof}

 \end{document}